\documentclass[10pt,oneside,reqno]{amsart}  
\usepackage{amssymb}
\usepackage{amsmath}
\usepackage{graphicx,color}
\usepackage{t1enc}
\usepackage[utf8]{inputenc}

\usepackage{amsthm}

\usepackage{latexsym}
\usepackage{subfigure}
\usepackage{tikz}
\usepackage{float}
\usepackage{graphicx,color}
\newtheorem{theorem}{Theorem}[section]
\newtheorem{definition}{Definition}[section]
\newtheorem{lemma}{Lemma}[section]
\newtheorem{proposition}{Proposition}[section]
\newtheorem{remark}{Remark}[section]
\newtheorem{example}{Example}[section]
\newtheorem{corollary}{Corollary}[section]

\newcommand\restr[2]{{
		\left.\kern-\nulldelimiterspace 
		#1 
		\vphantom{\big|} 
		\right|_{#2} 
}}

\begin{document}

\title[ positive solutions for singular quasilinear problems]{Existence, nonexistence and multiplicity of positive solutions for singular quasilinear problems}

\author{\bf\large Ricardo Lima Alves }\footnote{\thanks{The author was supported by Paraíba State Research Foundation (FAPESQ), Brazil: grant 1301/2021, Paraíba State Research Foundation (FAPESQ) }  } \hspace{2mm}
 

\maketitle
\begin{center}
		{\bf\small Abstract}
		
		\vspace{3mm} \hspace{.05in}\parbox{4.5in} {{\small In the present paper we deal with a quasilinear problem involving a singular term and a parametric superlinear perturbation.	We are interested in the existence, nonexistence and multiplicity of positive solutions as the parameter $\lambda>0$ varies. In our first result, the superlinear perturbation has an arbitrary growth and we obtain the existence of a solution for the problem by using the sub-supersolution method. For the second result, the superlinear perturbation has
				subcritical growth and we employ the Mountain Pass Theorem to show the existence of a second solution. } }
	\end{center}


%
%
%
\section{\bf Introduction}
\def\theequation{1.\arabic{equation}}\makeatother
\setcounter{equation}{0}

This paper is concerned with the existence, nonexistence and multiplicity of solutions for the family of quasilinear problems with singular nonlinearity
\begin{equation*}
\label{pq}\tag{$P_{\lambda}$}\left\{
\begin{array}{l}
-\Delta u -\Delta (u^{2})u=a(x) u^{-\gamma} + \lambda u^{p}~in ~ \Omega,\\
u> 0~in~ \Omega,\\ u(x)=0~~on~~\partial \Omega,
\end{array}
\right.
\end{equation*}
where $0<\gamma,3\leq p<\infty,0\leq\lambda$ is a parameter, $\Omega \subset \mathbb{R}^{N} (N\geq 3)$ is a bounded smooth domain and $a(x)$ is a positive measurable function.

We say that a function $u \in H_{0}^{1}(\Omega)\cap L^{\infty}(\Omega)$ is a weak solution (solution, for short) of \eqref{pq} if $u>0$ a.e. in $\Omega$, and, for every $\psi \in H^{1}_{0}(\Omega)$,
$$
	au^{-\gamma}\psi, u^{p}\psi \in L^{1}(\Omega)
$$
and
$$
\displaystyle \int_{\Omega} [(1+2u^{2})\nabla u \nabla \psi+2u\vert \nabla u\vert^{2}\psi] =\displaystyle \int_{\Omega} a(x)u^{-\gamma} \psi+\lambda\displaystyle \int_{\Omega}u^{p}\psi .
$$

Solutions of this type are related to the existence of standing wave solutions for quasilinear Schr\"odinger equations of the form
\begin{equation}\label{0}
i \partial_{t}z=-\Delta z +V(x)z+\eta(\vert z \vert^{2} )z-\kappa \Delta \rho(\vert z \vert^{2} )\rho^{\prime}(\vert z \vert^{2} )z,
\end{equation}
where $z:\mathbb{R}\times \Omega \to \mathbb{C}$, $V(x)$ is a given potential, $\kappa>0$ is a constant and $\eta,\rho$ are real functions. Quasilinear equations of form \eqref{0} appear more naturally in mathematical physics and have been derived as models of several physical phenomena corresponding to various types of $\rho$. The case of $\rho(s)=s$ was used for the superfluid film equation in plasma physics by Kurihara \cite{K} (cf. \cite{LSS}). In the case $\rho(s)=(1+s)^{1/2}$, equation \eqref{0} models the self-channeling of a high-power ultrashort laser in matter, see \cite{BG,BMMLB,CS,BR} and the references in \cite{BHS}. 

Consider the following quasilinear Schr\"odinger equation
\begin{equation}\label{00}
-\Delta u -\Delta (u^{2})u=g(x,u)~\mbox{in} ~ \Omega,
\end{equation}
where $\Omega \subset \mathbb{R}^{N}$ is a bounded domain with smooth boundary $\partial \Omega$. When $g$ is a singular nonlinearity, problems of type (\ref{00}) was studied by Do \'O--Moameni \cite{DM}, Liu--Liu--Zhao \cite{JDP}, Wang \cite{W}, Dos Santos--Figueiredo--Severo \cite{SFS}, Alves-Reis \cite{AR} and  Bal-Garain-Mandal- Sreenadh \cite{BGMS}. In particular,  the authors in \cite{JDP} considered the problem
\begin{equation}\label{000}
	\left\{
	\begin{array}{l}
		-\Delta_{s} u -\frac{s}{2^{s-1}}\Delta(u^{2})u=a(x)u^{-\gamma}+\lambda u^{p}~\mbox{in} ~ \Omega,\\
		u>0~\mbox{in} ~ \Omega,\\ u=0~\mbox{on}~\partial \Omega,
	\end{array}
	\right.
\end{equation}
where $N\geq 3$, $\Delta_{s}$ is the $s$-Laplacian operator, $2<2s<p+1<\infty$, $0<\gamma$ and $a \geq 0$ is a nontrivial measurable function satisfying the following condition: 
\begin{itemize}
	\item[(H)] There are $\varphi \in C_{0}^{1}(\overline{\Omega})$ and $q>N$ such that $\varphi>0$ on $\Omega$ and $a\varphi^{-\gamma} \in L^{q}(\Omega)$. 
\end{itemize}

 The authors used sub-supersolution method, truncation arguments and variational methods to prove the existence of solutions for \eqref{000} provided $\lambda>0$ is small enough.

In \cite{SFS},  Dos Santos--Figueiredo--Severo studied the problem
\begin{equation}\label{10}
\left\{
\begin{array}{l}
-\Delta u -\Delta (u^{2})u=a(x) u^{-\gamma}+z(x,u)~\mbox{in} ~ \Omega,\\
u>0~\mbox{in}~\Omega,\\
u=0~\mbox{on}~\partial \Omega,
\end{array}
\right.
\end{equation}
where $N\geq 3$, the function $a$ satisfies the hypothesis $(H)$ and the nonlinearity $z :\Omega \times \mathbb{R} \longrightarrow \mathbb{R}$ is continuous and satisfies (among other conditions):\\
There exist $C>0,r\geq 1$ and $b\in L^{\infty}(\Omega),b\geq 0$ almost everywhere in $\Omega$, such that 
$$\vert z(x,t)\vert \leq C(1+b(x)\vert t\vert^{r-1}),\forall t\in \mathbb{R}~\mbox{and a.e. in}~\Omega.$$

 By using sub-supersolution method, truncation arguments and the Mountain Pass Theorem they showed the existence of  solutions provided $\Vert b\Vert_{\infty}$ is small enough. When $z(x,t)=\lambda \vert t\vert^{r-2}t$ this is equivalent to the existence of  solutions for $\lambda>0$ small enough.

  In this paper, our first goal is to show the existence and nonexistence of solutions for \eqref{pq} without restriction on the parameter $\lambda$ and exponent $p\geq 3$.  We would like to emphasize that for $0<p<3$ the arguments carried out in \cite{AR,R} can be adapted to prove that problem \eqref{pq} has at least one solution for all $\lambda \in \mathbb{R}$.

It is worth pointing out that to prove our main results, we use the method of changing variables developed in Colin--Jeanjean \cite{CJ}. Thus, we reformulate problem \eqref{pq} into a new one which finds its natural setting in the Sobolev space $H_{0}^{1}(\Omega)$ (see problem \eqref{pa} in Section 2).

Our first result is the following.

\begin{theorem}\label{T1}
	Under the assumptions $(H)$ and $p\geq3$ there exists $0<\lambda_{\ast}<\infty$ such that problem \eqref{pq} has at least one solution $v_{\lambda}$ for $0<\lambda< \lambda_{\ast}$ and no solution for $\lambda > \lambda_{\ast}$. Moreover, $\lambda_{\ast}$ is characterized variationally by \eqref{ext} and $v_{\lambda}\in C_{0}^{1}(\overline{\Omega})$.
\end{theorem}

The proof of Theorem \ref{T1} is based on the method of sub-supersolutions. However, by virtue of the arbitrary growth of the  singular and superlinear terms  that appear in problem \eqref{pa}  we cannot use directly the method of sub-supersolutions here. An additional difficulty comes from the fact that these singular and superlinear terms are   nonhomogeneous. To overcome this difficulty we develop new arguments and a regularity result that allows us to obtain a subsolution to problem \eqref{pa} for all $\lambda>0$.  In particular, we  establish some preliminary results
 and we prove a sub-supersolution theorem (see Theorem \ref{TS}).

To prove the multiplicity of solutions for \eqref{pq}, with $\lambda \in (0,\lambda_{\ast})$, we need a refinement of hypotheses $(H)$.  We introduce the following assumption:
\begin{itemize}
	\item[$(H)_{\infty}$] There exists $\varphi \in C_{0}^{1}(\overline{\Omega})$  such that $\varphi>0$ on $\Omega$ and $a\varphi^{-1-\gamma} \in L^{\infty}(\Omega)$. 
\end{itemize}
If the function $\varphi$ satisfies $(H)_{\infty}$ then it satisfies $(H)$, too (see Section 4). 

  We denote by $2^{\ast}=2N/(N-2)$ the critical Sobolev exponent. Now we state our second result.

\begin{theorem}\label{T2}
	Under the assumptions $(H)_{\infty}$ and $3<p<22^{\ast}-1$, problem \eqref{pq} has at least two solutions  for $0<\lambda< \lambda_{\ast}$ and no solution for $\lambda > \lambda_{\ast}$. 
\end{theorem}

\begin{example}
	When $\Omega$ is the unit ball, the functions $a(x)=(1-\vert x\vert^{2})^{\sigma},\sigma \geq \gamma+1$ and $\varphi(x)=1-\vert x\vert^{2}$ satisfy assumption $(H)_{\infty}$.
\end{example}

Let us highlight that the hypotheses $(H)_{\infty}$ plays a crucial role in the proof of Theorem \ref{T2}. Indeed, it allows us to show that $v_{\lambda}$ is a local minimum of the functional $J_{\lambda}$ in the topology of $C_{0}^{1}(\overline{\Omega})$   and that the modified functional $\mathcal{J}_{\lambda}$ belongs to $C^{1}(H_{0}^{1}(\Omega),\mathbb{R})$ and satisfies the assumptions of Theorem 1 in Brezis-Nirenberg \cite{BN} (see \eqref{j} and \eqref{J} in Section 4 for definition of $J_{\lambda}$ and $\mathcal{J}_{\lambda}$, respectively ). In particular, we get that $v=0$ is a local minimum of the functional $\mathcal{J}_{\lambda}$ in the $H_{0}^{1}(\Omega)$ topology. Then, after fine arguments we apply the Mountain Pass Theorem  to obtain a second solution of \eqref{pq}. It is worth pointing out that under the assumption $(H)$ we are not able to show Lemma \ref{minl}  and that $\mathcal{J}_{\lambda}$ satisfies the assumptions of Theorem 1 in \cite{BN}.

We emphasize that Theorem \ref{T1} improve the works \cite{JDP,SFS} in the sense that we show the existence and nonexistence of solutions for \eqref{pq} without restriction on the parameter $\lambda$. They also did not prove a result of nonexistence of solutions. As far as we know, Theorem \ref{T2} is the first result of multiplicity of $H_{0}^{1}(\Omega)$-solutions for singular problems with strong singularity  $\gamma>1$ and without restriction on the parameter $\lambda$, that is, we do not assume $\lambda$ small enough. Notice that no restriction on the $\gamma>0$ is assumed.

There is a wide literature dealing with existence and multiplicity results for problems involving both the $p$-Laplacian operator and singular nonlinearities. The reader who wishes to broaden his/her knowledge on these topics is referred to \cite{SLS,H,HSS,PRR,PW,GST,GS}, and to the references therein.

The paper is structured as follows: In Section 2, we reformulate problem \eqref{pq} into a new one which finds its natural setting in the Sobolev space $H_{0}^{1}(\Omega)$ and we present some results that will be important for our work. In particular, we prove a nonexistence result and a sub-supersolution theorem. In Section 3, we prove Theorem \ref{T1} and Section 4 is devoted to prove Theorem \ref{T2}.  

\textbf{Notation.}  Throughout this paper, we make use of the following notations:
\begin{itemize}
		\item $L^{q}(\Omega)$, for $1\leq q \leq \infty$, denotes the Lebesgue space with usual norm denoted by $\Vert u\Vert_{q}$.
	\item $H_{0}^{1}(\Omega)$ denotes the Sobolev space endowed with inner product
	$$\left(u,v\right)=\displaystyle \int_{\Omega} \nabla u \nabla v,~\forall u,v \in H_{0}^{1}(\Omega). $$
	The norm associated with this inner product will be denoted by $\Vert~~ \Vert$. 
		\item $W_{0}^{2,q}(\Omega)$ denotes the Sobolev space with norm
		$$\Vert u \Vert=\left(\sum_{\vert \alpha\vert \leq 2}\Vert D^{\alpha }u\Vert_{q}^{q}\right)^{1/q}. $$
		 
		\item Let us consider the space $C^{1}_{0}(\overline{\Omega})=\left\{u\in C^{1}(\overline{\Omega}): u=0~\mbox{on}~ \partial \Omega \right\}$ equipped with the norm $\Vert u \Vert_{C^{1}}=\displaystyle\max_{x\in \Omega}\vert u(x)\vert+\displaystyle\max_{x\in \Omega}\vert \nabla u(x)\vert$. If on $C^{1}_{0}(\overline{\Omega})$ we consider the pointwise partial ordering (i.e., $u\leq v$ if and only if $u(x)\leq v(x)$ for all $x\in \overline{\Omega}$), which is induced by the positive cone
		$$C^{1}_{0}(\overline{\Omega})_{+}=\left\{u\in C^{1}_{0}(\overline{\Omega}): u\geq 0~\mbox{for all}~ x\in \Omega \right\},$$
		then this cone has a nonempty interior given by
		$$int (C^{1}_{0}(\overline{\Omega})_{+})=\left\{u\in C^{1}_{0}(\overline{\Omega}): u> 0~\mbox{for all}~ x\in \Omega~\mbox{and}~\frac{\partial u}{\partial \nu}(x) <0~\mbox{for all}~ x\in \partial\Omega\right\},$$
		where $\nu$ is the outward unit normal vector to $\partial \Omega$
		at the point $x\in \partial \Omega$.
		\item $B_{r}(v)$ denotes the ball centered at $v \in C_{0}^{1}(\overline{\Omega})$ with radius $r>0$ (with respect to the topology of $C_{0}^{1}(\overline{\Omega})$).
		\item The function $d(x)=d(x,\partial \Omega)$ denotes the distance from a point $x\in \overline{\Omega}$ to the boundary $\partial \Omega$, where $\overline{\Omega}=\Omega \cup \partial \Omega$ is the closure of $\Omega \subset \mathbb{R}^{N}$.
		\item We denote by $\phi_{1}$ the $L^{\infty}(\Omega)$-normalized (that is, $\Vert \phi_{1}\Vert_{\infty}=1$) positive eigenfunction for the smallest eigenvalue $\lambda_{1}>0$ of  $\left(-\Delta, H_{0}^{1}(\Omega)\right)$. 
		
		\item If $u$ is a measurable function, we denote  the positive  and negative parts by $u^{+}=\max\left\{u,0\right\}$ and $u^{-}=\max\left\{-u,0\right\}$, respectively. 
		\item If $A$ is a measurable set in $\mathbb{R}^{N}$, we denote by $\vert A \vert$ the Lebesgue measure of $A$. 
		\item $k,c,c_{1},c_{2},...$ and $C$ denote (possibly different from line to line) positive constants.
		\item The arrow $\rightharpoonup $(respectively, $\to $) denotes weak (respectively strong) convergence.
\end{itemize}

\section{Preliminaries}
In this section, we will establish some preliminaries which will be important for our work.  We reduce the study of the existence of positive solutions for \eqref{pq} to the existence of positive solutions of a singular elliptic problem. In particular, we will prove a nonexistence result and a sub-supersolution theorem.

 We denote by $\phi_{1}$ the $L^{\infty}(\Omega)$-normalized  positive eigenfunction for the smallest eigenvalue $\lambda_{1}>0$ of  $\left(-\Delta, H_{0}^{1}(\Omega)\right)$. We start by proving that $\phi_{1}$ satisfies the assumption $(H)$.
We consider the following assumption.
\begin{itemize}
	\item[$(H^{\prime})$]There is $q>N$ such that $a\phi_{1}^{-\gamma}\in L^{q}(\Omega)$.
\end{itemize}

\begin{lemma}\label{L0}
	Assumptions $(H)$ and $(H^{\prime})$ are equivalent.
\end{lemma}
\begin{proof}
Suppose that $(H)$ holds. One has $\phi_{1}\in int(C^{1}_{0}(\overline{\Omega})_{+})$ and $\varphi\in C^{1}_{0}(\overline{\Omega})_{+}$. Then,  from   Proposition 1 in \cite{MP} there exists $k>0$ such that $\phi_{1}\geq k\varphi$ in $\Omega$ and hence $ a\phi_{1}^{-\gamma}\leq k^{-\gamma}a\varphi^{-\gamma}\in L^{q}(\Omega)$, proving $(H^{'})$.

If $(H^{'})$ holds,  then the function $\varphi=\phi_{1}$ and $q$ satisfy $(H)$. This concludes the proof.
\end{proof}

\begin{remark}\label{rem1} \begin{itemize}
		\item[a)]The arguments in the proof of Lemma \ref{L0} can be used to prove that if $(H)$ holds, then any function $u\in int(C_{0}^{1}(\overline{\Omega})_{+})$ satisfies the assumption $(H)$, too. 
		\item[b)] If $\varphi$ satisfies the assumption $(H)$ then $ a\varphi^{1-\gamma},a \in L^{q}(\Omega)$. Indeed, $a=a\varphi^{-\gamma}\varphi^{\gamma}\leq \Vert \varphi \Vert_{\infty}^{\gamma}a\varphi^{-\gamma}\in L^{q}(\Omega)$ and $a\varphi^{1-\gamma}\leq  \Vert \varphi \Vert_{\infty}a\varphi^{-\gamma}\in L^{q}(\Omega)$.
		\item[c)] It is well known that $\phi_{1}\in C^{1}(\overline{\Omega})$ and satisfies $cd(x)\leq \phi_{1}(x)\leq Cd(x)$, $x\in \Omega$, for some constants $c, C>0$ (see \cite{V}). 
	\end{itemize}
	
\end{remark}

Now, we observe that the natural energy functional corresponding to the problem \eqref{pq} is the following: 
$$Q(u)=\frac{1}{2}\displaystyle\int_{\Omega} (1+2u^{2})|\nabla u|^{2}+\frac{1}{\gamma-1}\int_{\Omega} a(x)F(|u|)-\dfrac{\lambda}{p+1}\int_{\Omega} \vert u\vert^{p+1},~u\in A(Q),$$
where 
$$A(Q)=\left\{u\in H_{0}^{1}(\Omega): \displaystyle\int_{\Omega} a(x)F(|u|)<\infty~\mbox{and}~ \int_{\Omega} \vert u\vert^{p+1}<\infty\right\}$$
and the function $F:[0,\infty)\to [0,\infty]$ satisfies $F^{\prime}(s)= s^{-\gamma}$ for $s>0$ (see \cite{R} for a complete definition of $F$ ).

However, this functional is not well defined, because $\displaystyle\int_{\Omega} u^{2}|\nabla u|^{2}dx$ is not finite for all $u\in H^{1}_{0}(\Omega) $, hence it is difficult to apply variational methods directly. In order to overcome this difficulty, we use the following change of variables introduced by \cite{CJ}, namely, $v:=g^{-1}(u),$ where $g$ is defined by

\begin{equation}\label{g}
	\left\{
	\begin{array}{l}
		g^{\prime}(t)=\frac{1}{\left(1+2|g(t)|^{2}\right)^{\frac{1}{2}}}~~\mbox{in}~~ [0,\infty),\\
		g(t)=-g(-t)~~\mbox{in}~~ (-\infty,0].
	\end{array}
	\right.
\end{equation}

We list some properties of $g$, whose proofs can be found in \cite{L,AR,U,CJ}.
\begin{lemma}\label{L1} The function $g$ satisfies the following properties:
	\begin{itemize}
		\item[$(1)$] $g$ is uniquely defined, $C^{\infty}$ and invertible;
		\item[$(2)$] $g(0)=0$;
		\item[$(3)$] $0<g^{\prime}(t)\leq 1$ for all $t\in \mathbb{R}$;
		\item[$(4)$] $\frac{1}{2}g(t)\leq tg^{\prime}(t)\leq g(t)$ for all $t>0$;
		\item[$(5)$] $|g(t)|\leq |t|$ for all $t\in \mathbb{R}$;
		\item[$(6)$] $|g(t)|\leq 2^{1/4}|t|^{1/2}$ for all $t\in \mathbb{R}$;
		\item[$(7)$] $(g(t))^{2}-g(t)g^{\prime}(t)t\geq 0$ for all $t\in \mathbb{R}$;
		\item[$(8)$] There exists a positive constant $C$ such that $|g(t)|\geq C|t|$ for $|t|\leq 1$ and $|g(t)|\geq C|t|^{1/2}$ for $|t|\geq 1$;
		\item[$(9)$] $g^{\prime \prime}(t)<0$ when $t>0$ and $g^{\prime \prime}(t)>0$ when $t<0$;
		\item[$(10)$] the function $(g(t))^{1-\gamma}$ for $\gamma>1$ is decreasing for all $t>0$;
		\item[$(11)$] the function  $(g(t))^{-\gamma}$ is decreasing for all $t>0$;
		\item[$(12)$] $|g(t)g^{\prime}(t)|<1/ \sqrt[]{2}$ for all $t\in \mathbb{R}$;
		\item[$(13)$] $g^{2}(ts)\geq tg^{2}(s)$ for all $t\geq 1$ and $s\geq 0$.
	\end{itemize}
\end{lemma}

\begin{corollary}\label{C0} For each $s>0$ there exists a constant $K>0$ such that $\vert t^{\gamma}ln(g(t))\vert \leq K$ for all $0<t\leq s$.
\end{corollary}
\begin{proof}
	Since  $ h(t)=t^{\gamma}ln(g(t)),t>0$, is a continuous function it is sufficient to show that $\displaystyle \lim_{t\to 0}t^{\gamma}ln(g(t))=0$. From Lemma \ref{L1} $(8)$ one has
	$$ \vert t^{\gamma}ln(g(t))\vert \leq \vert C^{-\gamma}g^{\gamma}(t)ln(g(t))\vert ,$$
	for all $0<t\leq 1$, which implies that $\displaystyle \lim_{t\to 0}t^{\gamma}ln(g(t))=0$, because $\displaystyle \lim_{t\to 0}t^{\gamma}ln(t)=0$ and $\displaystyle \lim_{t\to 0}g(t)=0$.
\end{proof}

After a change of variable  $v=g^{-1}(u),$ we define an associated problem
\begin{equation*}\label{pa}\tag{$Q_{\lambda}$}
	\left\{
	\begin{array}{l}
		-\Delta v=\left[ a(x) (g(v))^{-\gamma} + \lambda(g(v))^{p}\right]g^{\prime}(v) ~\mbox{in} ~ \Omega,\\
		v> 0~\mbox{in}~ \Omega,\\
		v(x)=0 ~\mbox{on}~ \partial \Omega.
	\end{array}
	\right.
\end{equation*}

We say that a function $v \in H_{0}^{1}(\Omega)\cap L^{\infty}(\Omega)$ is a weak solution (solution, for short)  of \eqref{pa} if $v>0$ a.e. in $\Omega$, and, for every $\psi \in H^{1}_{0}(\Omega)$,
$$a(x)(g(v))^{-\gamma}g^{\prime}(v)\psi,(g(v))^{p}g^{\prime}(v)\psi \in L^{1}(\Omega)$$
and
$$
\int_{\Omega} \nabla v \nabla \psi =\int_{\Omega}  a(x) (g(v))^{-\gamma}g^{\prime}(v)\psi +\lambda\int_{\Omega} (g(v))^{p}g^{\prime}(v)\psi.$$

It is easy to see that problem \eqref{pa} is equivalent to our problem \eqref{pq}, which takes $u = g(v)$ as its solutions. Thus, our goal is reduced to proving the existence, nonexistence and multiplicity of solutions for the family of problems \eqref{pa}.

In order to study problem \eqref{pa}, one introduces 
the assumption:
\begin{itemize}
	\item[$(H)_{d}$] There are $\varphi \in C_{0}^{1}(\overline{\Omega})$ and $q>N$ such that $\varphi>0$ on $\Omega$ and $ag^{-\gamma}(\varphi)g^{\prime}(\varphi) \in L^{q}(\Omega)$ . 
\end{itemize}
The following lemma show the relation between the assumptions $(H)$ and $(H)_{d}$. 
\begin{lemma}\label{L000}
Suppose that the function $\varphi $ satisfies $(H)$. Then  $\varphi $ satisfies $(H)_{d}$. Moreover, $ag^{1-\gamma}(\varphi) \in L^{q}(\Omega)$ if $\gamma \neq 1$ and $a(x)ln(g(\varphi)) \in L^{q}(\Omega)$ if $\gamma = 1$.	
\end{lemma}
\begin{proof}
	Let $0<\epsilon<1$ such that $\epsilon \Vert \varphi \Vert_{\infty}< 1$ holds. By using $(3),(8),(9)$ and $(11)$ of Lemma \ref{L1} and Corollary \ref{C0} (if $\gamma=1$) we find 
	$$ag^{-\gamma}(\varphi)g^{\prime}(\varphi)\leq ag^{-\gamma}(\epsilon\varphi)g^{\prime}(\epsilon \varphi)\leq C^{-\gamma}\epsilon^{-\gamma}a\varphi^{-\gamma}\in L^{q}(\Omega),$$
 	$$ag^{1-\gamma}(\varphi)\leq g(\Vert \varphi \Vert_{\infty})ag^{-\gamma}(\epsilon\varphi)\leq g(\Vert \varphi \Vert_{\infty}) \epsilon^{-\gamma}C^{-\gamma}a\varphi^{-\gamma}  \in L^{q}(\Omega)$$
 	and 
 	$$\vert a(x)ln(g(\varphi))\vert =\vert a(x)\varphi^{-\gamma}\varphi^{\gamma}ln(g(\varphi))\vert \leq Ka(x)\varphi^{-\gamma}\in L^{q}(\Omega), $$
	namely, $ag^{-\gamma}(\varphi)g^{\prime}(\varphi) \in L^{q}(\Omega)$ and $ag^{1-\gamma}(\varphi) \in L^{q}(\Omega)$ and $a(x)ln(g(\varphi))\in L^{q}(\Omega)$ if $\gamma=1$.
\end{proof}

To prove the nonexistence of solutions for \eqref{pa} we define the function $m(x)=\min\left\{a(x),1\right\}\in L^{\infty}(\Omega)$ and we will denote by $\lambda_{1}[m]$ the principal eigenvalue of 
\begin{equation*}
\label{pq1}\tag{$A$}\left\{
\begin{array}{l}
-\Delta u =\lambda m(x)u~in ~ \Omega,\\
u(x)=0~~on~~\partial \Omega.
\end{array}
\right.
\end{equation*}

It is known that $\lambda_{1}[m]$ is simples, $\lambda_{1}[m]>0$, and the associated eigenfunction $\tilde{\phi}_{1}$ can be chosen such that $\tilde{\phi}_{1}>0$ in $\Omega$ (see \cite{GP}, Theorem 6.2.9).

Next, we prove the nonexistence  of positive solutions for \eqref{pa}.

\begin{lemma}\label{L2}
	There exists a constant $\lambda^{\ast}>0$ such that problem \eqref{pa} has no solution for all $\lambda\in  (\lambda^{\ast},\infty)$. 
\end{lemma}
\begin{proof}
	Let us start by defining the function  $j_{\lambda}(t)=(g^{-\gamma}(t)g^{\prime}(t)+\lambda g^{p}(t)g^{\prime}(t))/t$ for $t>0$. Using $(4)$ of Lemma \ref{L1} we have that
	$$j_{\lambda}(t)\geq \frac{g^{1-\gamma}(t)}{2t^{2}}+\lambda \frac{g^{p+1}(t)}{2t^{2}},t>0.$$
We now distinguish two cases:\\
{\it Case $\gamma>1$}. From $(5)$ and $(8)$ of Lemma \ref{L1} we get
\begin{equation}\label{1}
j_{\lambda}(t)\geq\left\{
\begin{array}{l}
 \dfrac{t^{-1-\gamma}}{2} + \lambda\dfrac{C^{p+1}t^{p-1}}{2} ~\mbox{if} ~ 0<t\leq 1,\\
\dfrac{t^{-1-\gamma}}{2} + \lambda\dfrac{C^{p+1}t^{(p-3)/2}}{2} ~\mbox{if} ~ t\geq 1.
\end{array}
\right.
\end{equation}

In order to find a lower bound for the function $j_{\lambda}$ we observe that the function $$\tilde{f}(t)=  \dfrac{t^{-1-\gamma}}{2} + \lambda\dfrac{C^{p+1}t^{p-1}}{2},t>0,$$

has a global minimizer 
$$t_{\lambda}=\left[\dfrac{(1+\gamma)}{\lambda(p-1)C^{p+1}}\right]^{\dfrac{1}{p+\gamma}},$$
such that $t_{\lambda}<1$ for $\lambda$ large enough and 
\begin{equation}\label{3}
\displaystyle \min_{t>0}\tilde{f}(t)=\tilde{f}(t_{\lambda})=\dfrac{1}{2}\left[\dfrac{\lambda(p-1)C^{p+1}}{1+\gamma}\right]^{\dfrac{1+\gamma}{p+\gamma}}\left(\dfrac{p+\gamma}{p-1}\right).
\end{equation}
Hereafter, we fix $\lambda$ such that $t_{\lambda}<1$. Then, by using \eqref{1} and \eqref{3}, we infer that
$$\displaystyle \min_{t>0}j_{\lambda}(t)\geq \min \left\{\dfrac{1}{2}\left[\dfrac{\lambda(p-1)C^{p+1}}{1+\gamma}\right]^{\dfrac{1+\gamma}{p+\gamma}}\left(\dfrac{p+\gamma}{p-1}\right),\lambda\dfrac{C^{p+1}}{2}\right\}$$
and as a consequence there exists $\lambda^{\ast}$ such that 
\begin{equation}\label{4}
j_{\lambda^{\ast}}(t_{\lambda^{\ast}}):=\displaystyle \min_{t>0}j_{\lambda^{\ast}}(t)\geq \lambda_{1}[m].
\end{equation}

{\it Case $\gamma\leq 1$}. From  $(8)$ of Lemma \ref{L1} we get
\begin{equation}\label{2}
j_{\lambda}(t)\geq\left\{
\begin{array}{l}
\dfrac{C^{1-\gamma}t^{-1-\gamma}}{2} + \lambda\dfrac{C^{p+1}t^{p-1}}{2} ~\mbox{if} ~ 0<t\leq 1,\\
\dfrac{C^{1-\gamma}t^{(-3-\gamma)/2}}{2} + \lambda\dfrac{C^{p+1}t^{(p-3)/2}}{2} ~\mbox{if} ~ t\geq 1.
\end{array}
\right.
\end{equation}

In order to find a lower bound for the function $j_{\lambda}$ we observe that the function $$\tilde{h}(t)=  \dfrac{C^{1-\gamma}t^{-1-\gamma}}{2} + \lambda\dfrac{C^{p+1}t^{p-1}}{2},t>0,$$

has a global minimizer 
$$t_{\lambda}=\left[\dfrac{(1+\gamma)}{\lambda(p-1)C^{p+\gamma}}\right]^{\dfrac{1}{p+\gamma}},$$
such that $t_{\lambda}<1$ for $\lambda$ large enough and 
\begin{equation}\label{5}
\displaystyle \min_{t>0}\tilde{h}(t)=\tilde{h}(t_{\lambda})=\dfrac{C^{2}}{2}\left[\dfrac{\lambda(p-1)}{1+\gamma}\right]^{\dfrac{1+\gamma}{p+\gamma}}\left[\dfrac{p+\gamma}{p-1}\right].
\end{equation}
Hereafter, we fix $\lambda$ such that $t_{\lambda}<1$. Then, by using \eqref{2} and \eqref{5}, we infer that
$$\displaystyle \min_{t>0}j_{\lambda}(t)\geq \min \left\{\dfrac{C^{2}}{2}\left[\dfrac{\lambda(p-1)}{1+\gamma}\right]^{\dfrac{1+\gamma}{p+\gamma}}\left(\dfrac{p+\gamma}{p-1}\right),\lambda\dfrac{C^{p+1}}{2}\right\}$$
and as a consequence there exists $\lambda^{\ast}$ such that 
\begin{equation}\label{6}
j_{\lambda^{\ast}}(t_{\lambda^{\ast}}):=\displaystyle \min_{t>0}j_{\lambda^{\ast}}(t)\geq \lambda_{1}[m].
\end{equation}

Now, arguing by contradiction, we  suppose that for some $\lambda>\lambda^{\ast}$ problem \eqref{pa} has a solution $v_{\lambda}$, where $\lambda^{\ast}$ is defined in \eqref{4} (if $\gamma> 1$) and \eqref{6}(if $\gamma \leq 1$). By taking $\tilde{\phi}_{1}$ as a test function in the equation satisfied by $v_{\lambda}$ and $v_{\lambda}$ in the equation satisfied by $\tilde{\phi}_{1}$ we obtain
 \begin{align*}
 \displaystyle \int(a(x)g^{-\gamma}(v_{\lambda})+\lambda^{\ast}g^{p}(v_{\lambda}))g^{\prime}(v_{\lambda})\tilde{\phi}_{1}&\geq \displaystyle \int m(x)(g^{-\gamma}(v_{\lambda})+\lambda^{\ast}g^{p}(v_{\lambda}))g^{\prime}(v_{\lambda})\tilde{\phi}_{1}\\
 \geq &\int m(x)j_{\lambda^{\ast}}(t_{\lambda^{\ast}})v_{\lambda}\tilde{\phi}_{1}\\
 \geq&\int \lambda_{1}[m]m(x)v_{\lambda}\tilde{\phi}_{1}\\
 =&\int \nabla \tilde{\phi}_{1}\nabla v_{\lambda}\\
 =&\displaystyle \int(a(x)g^{-\gamma}(v_{\lambda})+\lambda g^{p}(v_{\lambda}))g^{\prime}(v_{\lambda})\tilde{\phi}_{1}
 \end{align*}
 and hence $\lambda^{\ast}\geq \lambda$, which is impossible by the choice of $\lambda$. By virtue of the relation between \eqref{pq} and \eqref{pa} we deduce that problem \eqref{pq} has no solution for $\lambda >\lambda^{\ast}$. 
 \end{proof}

Now, we define the notions of subsolution and supersolution and prove a sub-supersolution theorem.
\begin{definition}
	We say that $v$ is a subsolution of problem \eqref{pa} if $v\in H_{0}^{1}(\Omega)\cap L^{\infty}(\Omega)$, $v>0$ in $\Omega$, $a(x) (g(v))^{-\gamma}g^{\prime}(v)\psi, (g(v))^{p}g^{\prime}(v)\psi \in L^{1}(\Omega)$  and 
	$$
	\int_{\Omega} \nabla v \nabla \psi \leq\int_{\Omega}  a(x) (g(v))^{-\gamma}g^{\prime}(v)\psi +\lambda\int_{\Omega} (g(v))^{p}g^{\prime}(v)\psi,$$
	for all $\psi\in H_{0}^{1}(\Omega), \psi \geq 0$ in $\Omega$. Similarly, $v\in H_{0}^{1}(\Omega)\cap L^{\infty}(\Omega)$, $v>0$ in $\Omega$, is a supersolution of \eqref{pa} if $a(x) (g(v))^{-\gamma}g^{\prime}(v)\psi, (g(v))^{p}g^{\prime}(v)\psi \in L^{1}(\Omega)$  and 
	$$
	\int_{\Omega} \nabla v \nabla \psi \geq\int_{\Omega}  a(x) (g(v))^{-\gamma}g^{\prime}(v)\psi +\lambda\int_{\Omega} (g(v))^{p}g^{\prime}(v)\psi,$$
	for all $\psi\in H_{0}^{1}(\Omega), \psi \geq 0$ in $\Omega$.
\end{definition}

\begin{theorem}\label{TS}
	Let $\underline{v}$ and $\overline{v}$ be a  subsolution respectively a  supersolution of problem \eqref{pa} such that $\underline{v}\leq \overline{v}$ in $\Omega$. Then there exists a solution $v\in H_{0}^{1}(\Omega)\cap L^{\infty}(\Omega)$ of \eqref{pa} such that $\underline{v}\leq v\leq \overline{v}$ in $\Omega$.
\end{theorem}
\begin{proof}
	We define a truncated function $\tilde{g}:\Omega \times \mathbb{R}\to \mathbb{R}$ by letting,
	\begin{equation*}
	\tilde{g}(x,t)=\left\{
	\begin{array}{l}
	 g^{p}(\underline{v}(x))g^{\prime}(\underline{v}(x)) ~\mbox{if} ~ t\leq \underline{v}(x),\\
	g^{p}(t)g^{\prime}(t) ~\mbox{if} ~ \underline{v}(x)\leq t\leq \overline{v}(x),\\
	g^{p}(\overline{v}(x))g^{\prime}(\overline{v}(x)) ~\mbox{if} ~  \overline{v}(x)\leq t.
	\end{array}
	\right.
	\end{equation*}
	Clearly, $\tilde{g}$ is a Carathéodory function. Moreover, $(3)$ and $(5)$ of Lemma \ref{L1} imply that
	\begin{equation}\label{66}
	\vert \tilde{g}(x,t)\vert \leq \vert \overline{v}(x)\vert^{p} \leq \Vert \overline{v}\Vert^{p}_{\infty}:=c,
	\end{equation}
	for all $(x,t)\in \Omega \times \mathbb{R}$. We denote by $\tilde{G}(x,t)=\displaystyle \int_{0}^{t}\tilde{g}(x,s)ds$ the primitive of $\tilde{g}$.

Now, we consider the auxiliary  singular elliptic problem 
\begin{equation*}\label{pau}\tag{$A_{\lambda}$}
\left\{
\begin{array}{l}
-\Delta v=a(x) (g(v))^{-\gamma}g^{\prime}(v) + \lambda\tilde{g}(x,v) ~\mbox{in} ~ \Omega,\\
v> 0~\mbox{in}~ \Omega,\\
v(x)=0 ~\mbox{on}~ \partial \Omega.
\end{array}
\right.
\end{equation*}

We will show that problem \eqref{pau} has a solution $v$ such that $\underline{v}\leq v\leq \overline{v}$ in $\Omega$. Thus, from definition of $\tilde{g}$ we obtain that $v$ is a solution of \eqref{pa}.  
Define the function $G$ as it follows:
\newline
if $0<\gamma<1$, $G(t)= \frac{g^{1-\gamma}(\vert t\vert)}{1-\gamma}$ and $t\in \mathbb{R}$,\\
if $\gamma=1$, \begin{equation*}
	G(t)=\left\{
	\begin{array}{l}
		\ln g(t),~if ~ t> 0\\
		+\infty,~if~ t=0
	\end{array}
	\right.
\end{equation*}
if $\gamma>1$,
\begin{equation*}
	G(t)=\left\{
	\begin{array}{l}
		\frac{g^{1-\gamma}(t)}{1-\gamma},~if ~ t> 0\\
		+\infty,~if~ t=0.
	\end{array}
	\right.
\end{equation*}

We can associate to problem \eqref{pau} the following energy functional
\begin{equation}\label{fun}
	I_{\lambda}(v)=\frac{1}{2}\Vert v\Vert^{2}-\displaystyle \int_{\Omega}a(x) G(\vert v\vert)-\lambda \displaystyle \int_{\Omega}\tilde{G}(x,v),
\end{equation}
for every $v\in D$, where  
\begin{equation}\label{EF}
D=\left\{v\in H^{1}_{0}(\Omega): \displaystyle \int_{\Omega}a(x)G(\vert v\vert)\in \mathbb{R} \right\}
\end{equation}
is the effective domain of $I_{\lambda}$. As we known, the functional $I_{\lambda}$ fails to be G\^ateaux differentiable because of the singular term, then we can not apply the critical point theory for functionals of class $C^{1}$.

The assumption $(H)$ and Lemmas \ref{L0} and \ref{L000} imply that $aG(\phi_{1})\in L^{q}(\Omega)$. In particular, one has $\phi_{1}\in D$ and hence $D\neq \emptyset$. Then, using \eqref{66} and arguing as in the proof of Theorems 1.1 and 1.2 of  \cite{R} we can show that there exists a solution $v$ of \eqref{pau} and it satisfies 
$$I_{\lambda}(v)=\displaystyle \inf_{z\in D}I_{\lambda}(z).$$

It remains to check that $\underline{v}\leq v\leq \overline{v}$ in $\Omega$. We set $(v-\underline{v})^{-}=\max\left\{-(v-\underline{v}),0\right\}$. Using that $\underline{v}$ is a subsolution and $v$ is a solution, we have
	$$
	\int_{\Omega} \nabla \underline{v} \nabla (v-\underline{v})^{-} \leq\int_{\Omega}  a(x) (g(\underline{v}))^{-\gamma}g^{\prime}(\underline{v})(v-\underline{v})^{-} +\lambda\int_{\Omega} (g(\underline{v}))^{p}g^{\prime}(\underline{v})(v-\underline{v})^{-},$$
$$
\int_{\Omega} \nabla v \nabla (v-\underline{v})^{-} =\int_{\Omega}  a(x) (g(v))^{-\gamma}g^{\prime}(v)(v-\underline{v})^{-} +\lambda\int_{\Omega} \tilde{g}(x,v)(v-\underline{v})^{-},$$
and applying $(9),(10)$ and $(11)$ of Lemma \ref{L1}, we find
\begin{align*}
-\int_{\Omega} \vert \nabla (v-\underline{v})^{-}\vert^{2}&\geq\int_{\left\{v<\underline{v}\right\}}  a(x) ((g(v))^{-\gamma}g^{\prime}(v)-(g(\underline{v}))^{-\gamma}g^{\prime}(\underline{v}  ))(v-\underline{v})^{-}\\& +\lambda\int_{\left\{v<\underline{v}\right\}} (\tilde{g}(x,v)-(g(\underline{v}))^{p}g^{\prime}(\underline{v}))(v-\underline{v})^{-}\\
&\geq \lambda\int_{\left\{v<\underline{v}\right\}} (\tilde{g}(x,v)-(g(\underline{v}))^{p}g^{\prime}(\underline{v}))(v-\underline{v})^{-}\\
&= \lambda\int_{\left\{v<\underline{v}\right\}} ((g(\underline{v}))^{p}g^{\prime}(\underline{v})-(g(\underline{v}))^{p}g^{\prime}(\underline{v}))(v-\underline{v})^{-}\\
&=0,
\end{align*}
namely $\Vert (v-\underline{v})^{-}\Vert=0$, which means that $\underline{v}\leq v$ in $\Omega$.

Similarly, setting $(v-\overline{v})^{+}=\max\left\{v-\overline{v},0\right\}$ and using that $\overline{v}$ is a supersolution and $v$ is a solution, jointly with $(9),(10)$ and $(11)$ of Lemma \ref{L1}, we get
\begin{align*}
\int_{\Omega} \vert \nabla (v-\overline{v})^{+}\vert^{2}&\leq\int_{\left\{\overline{v}<v\right\}}  a(x) ((g(v))^{-\gamma}g^{\prime}(v)-(g(\overline{v}))^{-\gamma}g^{\prime}(\overline{v}  ))(v-\overline{v})^{+}\\& +\lambda\int_{\left\{\overline{v}<v\right\}} (\tilde{g}(x,v)-(g(\overline{v}))^{p}g^{\prime}(\overline{v}))(v-\overline{v})^{+}\\
&\leq \lambda\int_{\left\{\overline{v}<v\right\}} (\tilde{g}(x,v)-(g(\overline{v}))^{p}g^{\prime}(\overline{v}))(v-\overline{v})^{+}\\
&= \lambda\int_{\left\{\overline{v}<v\right\}} ((g(\overline{v}))^{p}g^{\prime}(\overline{v})-(g(\overline{v}))^{p}g^{\prime}(\overline{v}))(v-\overline{v})^{+}\\
&=0,
\end{align*}
namely $\Vert (v-\overline{v})^{+}\Vert=0$, which means that $v\leq \overline{v}$ in $\Omega$. This completes the proof of the theorem.

\end{proof}

\begin{remark}\label{rem2}
	\begin{itemize}
		\item[a)] Arguing as in the proof of  Lemmas \ref{L0} and \ref{L000} we can show that $int(C^{1}_{0}(\overline{\Omega})_{+}) \subset D$ $(see ~~\eqref{EF})$. Hence it makes sense to consider the local minimum obtained in Lemma \ref{minl}, because $v_{\lambda}\in int(C^{1}_{0}(\overline{\Omega})_{+}) \subset D$. 
		\item[b)] If $0<\gamma<1$ holds, then $I_{\lambda}(v)<0$. Indeed, applying Lemma \ref{L2} $(8)$ we obtain
		$$I_{\lambda}(v)\leq I_{\lambda}(t\phi_{1}) \leq \frac{t^{2}}{2}\Vert \phi_{1}\Vert^{2}-\dfrac{C^{1-\gamma}t^{1-\gamma}}{1-\gamma}\displaystyle \int_{\Omega}a(x)\phi_{1}^{1-\gamma}<0,$$
		provided $0<t<1$ is small enough.
	\end{itemize}
	
\end{remark}

The following lemma shows the existence of a subsolution of \eqref{pa} for all $\lambda>0$.
\begin{lemma}\label{L00}
	If $v_{0}\in H_{0}^{1}(\Omega)$ is the unique weak solution of $(Q_{0})$, then $v_{0}\in C^{1}_{0}(\overline{\Omega})$ and $v_{0}(x)\geq C d(x)$ in $\Omega$ for some constant $C>0$. Moreover, $a(x) (g(v_{0}))^{-\gamma}g^{\prime}(v_{0})\in L^{q}(\Omega)$ and $v_{0}$ is a subsolution of \eqref{pa} for all $\lambda>0$.
\end{lemma}
\begin{proof}
	From Lemma \ref{L0} and Remark \ref{rem1} $b)$ one has $a(x)\phi_{1}^{1-\gamma}\in L^{q}(\Omega),q>1$ and hence, the existence of a unique weak solution $v_{0}\in H_{0}^{1}(\Omega)$  of $(Q_{0})$ follows from Theorem 1.3 in  \cite{AR}. Now we want to show that $v_{0}\in C_{0}^{1}(\overline{\Omega})$. Using Theorem 3 of Brezis-Nirenberg \cite{BN} there exist constants $c_{1},c_{2}>0$ such that $v_{0}(x) \geq c_{2}d(x)\geq c_{1}\phi_{1}(x)$ in $\Omega$ and $c_{1}\phi_{1}(x)<1$ in $\Omega$. By Lemma \ref{L1} $(3),(8),(11)$ and Lemma \ref{L0},
	$$ a(x) (g(v_{0}))^{-\gamma}g^{\prime}(v_{0})\leq C^{-\gamma}c_{1}^{-\gamma}a(x)\phi_{1}^{-\gamma}\in L^{q}(\Omega),$$
	that is, $a(x) (g(v_{0}))^{-\gamma}g^{\prime}(v_{0})\in L^{q}(\Omega)$ with $q>N$. Thus, by elliptic regularity, $v_{0}\in W_{0}^{2,q}(\Omega)$, and then by the Sobolev embedding theorem we have $v_{0}\in C^{1}_{0}(\overline{\Omega})$. Finally, from the fact that $v_{0}$ is a solution of $(Q_{0})$ and $v_{0}\in C_{0}^{1}(\overline{\Omega})$ one deduces that $v_{0}$ is a subsolution of \eqref{pa} for all $\lambda>0$. This completes the proof.
\end{proof}

We end this section with the following lemma.
\begin{lemma}\label{L4}
Let $v\in H_{0}^{1}(\Omega),v>0$ in $\Omega$, and 	suppose that
	$$
	\int_{\Omega} \nabla v \nabla \psi \geq\int_{\Omega}  a(x) (g(v))^{-\gamma}g^{\prime}(v)\psi +\lambda\int_{\Omega} (g(v))^{p}g^{\prime}(v)\psi,$$
	for all $\psi\in C_{0}^{1}(\overline{\Omega}), \psi \geq 0$, holds. Then 
	$$
	\int_{\Omega} \nabla v \nabla \psi \geq\int_{\Omega}  a(x) (g(v))^{-\gamma}g^{\prime}(v)\psi +\lambda\int_{\Omega} (g(v))^{p}g^{\prime}(v)\psi,$$
	for all $\psi\in H_{0}^{1}(\Omega), \psi \geq 0$ in $\Omega$, holds. In particular, $v\geq v_{0}$ in $\Omega$, where $v_{0}$ is the unique solution of $(Q_{0})$.
\end{lemma}
\begin{proof}
	Let $\psi\in H_{0}^{1}(\Omega), \psi \geq 0$ in $\Omega$, then from the proof of Theorem 4.4 of \cite{C} there exists $\psi_{n}\in C_{0}^{\infty}(\overline{\Omega}),\psi_{n}\geq 0$ such that $\psi_{n}\to \psi $ in $H_{0}^{1}(\Omega)$ and $\psi_{n}\to \psi $ a.e. in $\Omega$. Hence, 
	$$
	\int_{\Omega} \nabla v \nabla \psi_{n} \geq\int_{\Omega}  a(x) (g(v))^{-\gamma}g^{\prime}(v)\psi_{n} +\lambda\int_{\Omega} (g(v))^{p}g^{\prime}(v)\psi_{n},$$
	and using the Fatou lemma we deduce that
	$$
	\int_{\Omega} \nabla v \nabla \psi \geq\int_{\Omega}  a(x) (g(v))^{-\gamma}g^{\prime}(v)\psi +\lambda\int_{\Omega} (g(v))^{p}g^{\prime}(v)\psi,$$
	proving the first statement of the lemma. 
	
	It remains to show that $v\geq v_{0}$ in $\Omega$. For this, we take $(v-v_{0})^{-}$ as a test function in the equation satisfied by
	$v_{0}$ and in the inequality  satisfied by  $v$, and arguing as in Theorem \ref{TS} one finds $v\geq v_{0}$ in $\Omega$. The proof is complete. 
\end{proof}

\section{Proof of Theorem \ref{T1}}
This section is devoted to the proof of Theorem \ref{T1}. In the rest of this paper we will use the same notation introduced in the previous section.

Let us define
\begin{align*}
\mathcal{L}&=\left\{\lambda>0: \mbox{problem}~\eqref{pa}~\mbox{has at least one solution}\right\}\\
&=\left\{\lambda>0: \mbox{problem}~\eqref{pq}~\mbox{has at least one solution}\right\}
\end{align*}
and set
$$\Lambda=\sup \mathcal{L}.$$

We start by proving the following lemma.
\begin{lemma}\label{L3}
	The set $\mathcal{L}$ is nonempty and $\Lambda$ is finite.
\end{lemma}
\begin{proof}
Let $\underline{v}=v_{0}$  and consider the problem
\begin{equation*}\label{pauu}\tag{$T$}
\left\{
\begin{array}{l}
-\Delta v=a(x) (g(\underline{v}))^{-\gamma}g^{\prime}(\underline{v}) + 1 ~\mbox{in} ~ \Omega,\\
v> 0~\mbox{in}~ \Omega,\\
v(x)=0 ~\mbox{on}~ \partial \Omega.
\end{array}
\right.
\end{equation*}
By using Lemma \ref{L00} we infer that $a(x) (g(\underline{v}))^{-\gamma}g^{\prime}(\underline{v}) + 1 \in L^{q}(\Omega)$. Therefore problem \eqref{pauu} has a solution $\overline{v}\in W^{2,q}(\Omega)$ and by the Sobolev embedding theorem, $\overline{v}\in C^{1}_{0}(\overline{\Omega})$. Moreover,
$$-\Delta \overline{v}\geq a(x) (g(\underline{v}))^{-\gamma}g^{\prime}(\underline{v})=-\Delta \underline{v}~\mbox{in}~\Omega,$$
which implies that $\overline{v}\geq\underline{v}$ in $\Omega$. From this and Lemmas \ref{L1} $(9),(11)$ and  \ref{L00} we get that  
$$
\int_{\Omega} \nabla \overline{v} \nabla \psi \geq \int_{\Omega}  a(x) (g(\overline{v}))^{-\gamma}g^{\prime}(\overline{v})\psi +\lambda\int_{\Omega} (g(\overline{v}))^{p}g^{\prime}(\overline{v})\psi,$$
for all $\psi\in H_{0}^{1}(\Omega), \psi \geq 0$ in $\Omega$, and for $\lambda>0$ satisfying $\lambda \Vert (g(\overline{v}))^{p}g^{\prime}(\overline{v})\Vert_{\infty}\leq 1$. For such values of $\lambda$, we can apply  Theorem \ref{TS} to deduce the existence of a solution $v$ of \eqref{pa} such that $\underline{v}\leq v\leq \overline{v}$ in $\Omega$ (and consequently $v\in L^{\infty}(\Omega)$). Therefore
 $ \mathcal{L}\neq \emptyset$. 
 
 By Lemma \ref{L2} we obtain that $\Lambda$ is finite. The proof is complete.
\end{proof}
Following \cite{IY} we introduce
\begin{equation}\label{ext}
\lambda_{\ast}=\displaystyle \sup_{v\in S}\inf_{\psi \in \Phi}\left\{L(v,\psi)\right\}
\end{equation}
where 
$$L(v,\psi):=\dfrac{\displaystyle \int_{\Omega}\nabla v\nabla \psi-\displaystyle \int_{\Omega}a(x)(g(v))^{-\gamma}g^{\prime}(v)\psi}{\displaystyle \int_{\Omega}(g(v))^{p}g^{\prime}(v)\psi}$$
 is the extended functional  and
$$\Phi =\left\{\psi \in C^{1}_{0}(\overline{\Omega})\backslash\{0\} : \psi \geq 0 ~\mbox{in}~\Omega \right\},$$
$$S=\left\{v\in H_{0}^{1}(\Omega)\cap L^{\infty}(\Omega):v\geq C(v)d(x)~\mbox{in}~\Omega  \right\},$$
where $0<C(v)<\infty$ is a positive constant which can depend on $v$. If $v \in S$ then $v\geq k\phi_{1}$ in $\Omega$ for some $k>0$ (see Remark \ref{rem1} $c)$), and from Lemmas \ref{L0}, \ref{L1} and \ref{L000} it follows that $L$ is well defined. 

Some properties of $\lambda_{\ast}$ are stated in the following theorem.
\begin{theorem}\label{ast}The following properties hold true:
\begin{itemize}
	\item[$a)$] $0<\lambda_{\ast}<\infty$.
	\item[$b)$] $\lambda_{\ast}=\Lambda$.
\end{itemize}
\end{theorem}
\begin{proof}
	$a)$ From Lemma \ref{L3} there exist $\lambda>0$ and $v\in H_{0}^{1}(\Omega)\cap L^{\infty}(\Omega), v>0$ in $\Omega$, such that 
$$
\int_{\Omega} \nabla v \nabla \psi =\int_{\Omega}  a(x) (g(v))^{-\gamma}g^{\prime}(v)\psi +\lambda\int_{\Omega} (g(v))^{p}g^{\prime}(v)\psi,$$
for all $\psi \in H_{0}^{1}(\Omega)$, which together with  Theorem 3 of Brezis-Nirenberg \cite{BN} implies that $v\in S$ and $0<\lambda=L(v,\psi)$ for all $\psi \in \Phi$. As a consequence we get
$$0<\lambda=\displaystyle \inf_{\psi \in \Phi}\left\{L(v,\psi)\right\}\leq \lambda_{\ast}.$$

To prove that $\lambda_{\ast}< \infty$, we argue by contradiction. Assume that $\lambda_{\ast}=\infty$. Then, by the definition of $\lambda_{\ast}$ there exists $v\in S$ such that $\Lambda <\lambda:=\displaystyle \inf_{\psi \in \Phi}\left\{L(v,\psi)\right\} $, that is,
$$
\int_{\Omega} \nabla v \nabla \psi \geq\int_{\Omega}  a(x) (g(v))^{-\gamma}g^{\prime}(v)\psi +\lambda\int_{\Omega} (g(v))^{p}g^{\prime}(v)\psi,$$
for all $\psi \in \Phi$. By using Lemma \ref{L4} we deduce that $v$ is a supersolution of \eqref{pa} and $v\geq v_{0}$ in $\Omega$. Moreover, from Lemma \ref{L00} one has that $v_{0}$ is a subsolution of \eqref{pa}. As a consequence we can apply Theorem \ref{TS}, with $\underline{v}=v_{0}$ and $\overline{v}=v$, to deduce the existence of a solution of \eqref{pa}, which implies $\lambda \leq \Lambda$, contradicting the fact that $\lambda >\Lambda$. Therefore $\lambda_{\ast}<\infty$.  

$b)$ Let $v\in S$ such that $0<\lambda=\displaystyle \inf_{\psi \in \Phi}\left\{L(v,\psi)\right\}$. Arguing as in $a)$ we can prove that problem \eqref{pa} has a solution, namely, $\lambda \in \mathcal{L}$ and since $\lambda$ is arbitrary, we have $\lambda_{\ast}=\displaystyle \sup_{v\in S}\inf_{\psi \in \Phi}\left\{L(v,\psi)\right\}\leq \Lambda$. We claim that $\lambda_{\ast}=\Lambda$. Otherwise, $\lambda_{\ast}<\Lambda$ and by the definition of $\Lambda$ there exists $\lambda>\lambda_{\ast}$ such that problem \eqref{pa} has a solution $v$. Again, arguing as in $a)$ we find that $v\in S$ and $\lambda=\displaystyle \inf_{\psi \in \Phi}\left\{L(v,\psi)\right\}\leq \lambda_{\ast}$,  contradicting the fact that $\lambda >\lambda_{\ast}$. Therefore $\lambda_{\ast}=\Lambda$. This finishes the proof.

\end{proof}

We are now in position to prove Theorem \ref{T1}.

\textbf{Proof of Theorem \ref{T1}} Let us show that problem \eqref{pa} has a solution for $\lambda \in (0,\lambda_{\ast})$ and no solution for $\lambda \in (\lambda_{\ast},\infty)$, where $\lambda_{\ast}$ is defined in \eqref{ext}. Let $\lambda \in (0,\lambda_{\ast})$. Then, by the definition of $\lambda^{\ast}$, there exists $z\in S$ such that $\lambda \leq L(z,\psi)$ for all $\psi\in \Phi$. We deduce from this inequality and Lemma \ref{L4} that $z$ is a supersolution of \eqref{pa} with $z\geq v_{0}$ in $\Omega$. Applying Theorem \ref{TS} with $\underline{v}=v_{0}$ and $\overline{v}=z$ we obtain that problem \eqref{pa} has a solution $v$ with $\underline{v}\leq v\leq \overline{v}$ in $\Omega$. To show that $v\in C_{0}^{1}(\overline{\Omega})$ we follow \cite{AR}. By Lemma \ref{L1} $(3),(5),(9),(11)$ and Lemma \ref{L00} we infer
$$a(x)g^{-\gamma}(v)g^{\prime}(v)\leq a(x)g^{-\gamma}(\underline{v})g^{\prime}(\underline{v})\in L^{q}(\Omega)$$
and
$$g^{p}(v)g^{\prime}(v)\leq\vert \overline{v}\vert^{p} \leq \Vert \overline{v}\Vert^{p}_{\infty}\in L^{\infty}(\Omega)$$
and as a consequence there exist $z_{1},z_{2}\in C_{0}^{1,\alpha}(\overline{\Omega})$, for some $\alpha \in (0,1)$, satisfying 
$$
\int_{\Omega} \nabla z_{1} \nabla \psi =\int_{\Omega}  a(x) (g(v))^{-\gamma}g^{\prime}(v)\psi ~\mbox{and}~\int_{\Omega} \nabla z_{2} \nabla \psi =\lambda\int_{\Omega} (g(v))^{p}g^{\prime}(v)\psi,$$
for all $\psi \in H_{0}^{1}(\Omega)$. From this we get
$$
\int_{\Omega} \nabla v \nabla \psi=\int_{\Omega} \nabla z_{1} \nabla \psi +\int_{\Omega} \nabla z_{2} \nabla \psi,$$
for all $\psi \in H_{0}^{1}(\Omega)$, which implies $v=z_{1}+z_{2}$, and hence $v\in C_{0}^{1,\alpha}(\overline{\Omega})$. Furthermore,  by the strong maximum principle and the Hopf lemma we find that $v\in int (C_{0}^{1}(\overline{\Omega})_{+})$.

Finally, from Theorem \ref{ast} we have $\lambda_{\ast}=\Lambda$ and by the definition of $\Lambda$ problem \eqref{pa} has no solution for $\lambda>\lambda_{\ast}=\Lambda$. This completes the proof of the theorem.

\section{Proof of Theorem \ref{T2}}

In this section we are going to prove Theorem \ref{T2}. In order to do this, we adapt the arguments carried out in \cite{ABC}. From now on, we will assume $(H)_{\infty}$ and $3<p< 22^{\ast}-1$ hold. Proceeding
as in Section 1  we can prove that:
\begin{itemize}
	\item $a\phi_{1}^{-1-\gamma},a\phi_{1}^{-\gamma}  \in L^{\infty}(\Omega)$.
	\item $ag^{-1-\gamma}(\varphi)g^{\prime}(\varphi), ag^{-1-\gamma}(\phi_{1})g^{\prime}(\phi_{1}) \in L^{\infty}(\Omega)$ and $ag^{1-\gamma}(\varphi) \in L^{\infty}(\Omega)$ if $\gamma \neq 1$, and $a(x)ln(g(\varphi))\in L^{\infty}(\Omega)$ if $\gamma=1$.
	\item if $v_{\lambda}$ is the solution obtained in Theorem \ref{T1}, then 
	 \begin{equation}\label{41}
	a(x)g^{-\gamma}(v_{\lambda})g^{\prime}(v_{\lambda})\in L^{\infty}(\Omega).
	\end{equation} 
\end{itemize}

We start by defining the functional 
\begin{equation}\label{j}
J_{\lambda}(v)=\frac{1}{2}\Vert v\Vert^{2}-\displaystyle \int_{\Omega}a(x) G(\vert v\vert)-\frac{\lambda}{p+1} \displaystyle \int_{\Omega}g^{p+1}(v), v\in D.
\end{equation}
It is worth recalling that $int (C_{0}^{1}(\overline{\Omega})_{+})\subset D$ (see \eqref{EF} and Remark \ref{rem2} ).  The functional $J_{\lambda}$ fails to be Fr\'echet differentiable in $H_{0}^{1}(\Omega)$ because of the singular term, then critical point theory could not be applied to obtain the existence of solutions directly.

 In this section, we denote by $I_{\lambda}$ the functional defined in \eqref{fun} of Theorem \ref{TS}.

An important property of the solution obtained in Theorem \ref{T1} is the following.

\begin{lemma}\label{minl}
	Let $0<\lambda<\lambda_{\ast}$. If $v_{\lambda}$ is the solution of \eqref{pa}  obtained in Theorem \ref{T1}, then $v_{\lambda}$ is a local minimum of $J_{\lambda}$ in the $C^{1}_{0}(\overline{\Omega})$ topology.
\end{lemma}
\begin{proof}
	Without loss of generality, we can assume that $\overline{v}$ is a solution of $(Q_{\mu})$ for some $\mu \in (\lambda,\lambda^{\ast})$. Hence, arguing as in Theorem \ref{T1} one has $\overline{v}\in C_{0}^{1}(\overline{\Omega})$ and by the strong maximum principle and the Hopf lemma we infer that $\overline{v}\in int (C_{0}^{1}(\overline{\Omega})_{+})$. Now, the proof is based on the following claims.\\
	${\it Claim ~1}$. $\overline{v}-v_{\lambda}\in int (C_{0}^{1}(\overline{\Omega})_{+})$. We have
	\begin{align*}
	-\Delta(\overline{v}- v_{\lambda})&\geq a(x) ((g(\overline{v}))^{-\gamma}g^{\prime}(\overline{v})-(g(v_{\lambda}))^{-\gamma}g^{\prime}(v_{\lambda}))\\& + \lambda(g(\overline{v}))^{p}g^{\prime}(\overline{v})-(g(v_{\lambda}))^{p}g^{\prime}(v_{\lambda}))
	\end{align*}
	and by the mean value theorem there exist measurable functions $\theta_{1}(x)$ and $ \theta_{2}(x)$ such that $v_{\lambda}(x)\leq \theta_{1}(x), \theta_{2}(x)\leq \overline{v}(x)$ for all $x\in \Omega$ and
	\begin{align}\label{40}
	-\Delta(\overline{v}- v_{\lambda})&\geq a(x) ((g(\theta_{1}(x)))^{-\gamma}g^{\prime}(\theta_{1}(x)))^{\prime}(\overline{v}(x) -v_{\lambda}(x))\\
	& + \lambda((g(\theta_{2}(x))))^{p}g^{\prime}(\theta_{2}(x))))^{\prime}(\overline{v}(x) -v_{\lambda}(x)). \nonumber
	\end{align}

	From the definition of $g^{\prime}$ and Lemma \ref{L1} $(3)$, it follows that 
	
	\begin{equation}\label{gi}
	(g^{-\gamma}(t)g^{\prime}(t))^{\prime}\geq -g^{-1-\gamma}(t)(\gamma+2g^{2}(t)),t>0,
	\end{equation}
	
	\begin{equation*}
	\vert (g^{p}(t)g^{\prime}(t))^{\prime}\vert \leq pg^{p-1}(t)+2g^{p+1}(t),t>0,
	\end{equation*}
	hold. Then, again by Lemma \ref{L1} $(3),(11)$ one has
	$$(g^{-\gamma}(\theta_{1}(x))g^{\prime}(\theta_{1}(x)))^{\prime}\geq -g^{-1-\gamma}(v_{\lambda}(x))(\gamma+2g^{2}(\Vert \overline{v}\Vert_{\infty})),$$
	$$\vert (g^{p}(\theta_{2}(x))g^{\prime}(\theta_{2}(x)))^{\prime}\vert \leq pg^{p-1}(\Vert \overline{v}\Vert_{\infty})+2g^{p+1}(\Vert \overline{v}\Vert_{\infty}),$$
	for all $x\in \Omega$. We set
	$$c_{1}=\Vert ag^{-1-\gamma}(v_{\lambda})\Vert_{\infty}(\gamma+2g^{2}(\Vert \overline{v}\Vert_{\infty})), c_{2}=p\lambda g^{p-1}(\Vert \overline{v}\Vert_{\infty})+2\lambda g^{p+1}(\Vert \overline{v}\Vert_{\infty})$$
	and $c=c_{1}+c_{2}$. With these estimates and definitions, in view of \eqref{40}, we get
	$$-\Delta(\overline{v}- v_{\lambda})\geq(-c_{1}-c_{2})(\overline{v}- v_{\lambda}) =-c(\overline{v}- v_{\lambda})$$
	that is
	$$-\Delta(\overline{v}- v_{\lambda})+c(\overline{v}- v_{\lambda})\geq 0~\mbox{in}~\Omega,$$
	and since $\overline{v}- v_{\lambda}\neq 0$, we can apply Theorem 3 of \cite{BN} to deduce the existence of constants $c_{3},c_{4}>0$ such that
	$$\overline{v}- v_{\lambda}\geq c_{3}d(x)\geq c_{4}\phi_{1}(x)~\mbox{in}~\Omega.$$
	
	As a consequence we obtain
	$$\dfrac{\partial(\overline{v}- v_{\lambda})}{\partial \nu}\leq c_{4}\dfrac{\partial \phi_{1}}{\partial \nu}<0~\mbox{on}~\partial \Omega,$$
	which jointly with $\overline{v}-v_{\lambda}>0$  in $\Omega$ means that  $\overline{v}-v_{\lambda}\in int (C_{0}^{1}(\overline{\Omega})_{+})$, and this proves the claim 1. \\
	${\it Claim ~2}$. $v_{\lambda}-\underline{v}\in int (C_{0}^{1}(\overline{\Omega})_{+})$. The proof is essentially equal to the one of Claim 1. Indeed, we set 
$$c_{1}=\Vert ag^{-1-\gamma}(\underline{v})\Vert_{\infty}(\gamma+2g^{2}(\Vert \overline{v}\Vert_{\infty})),$$
and from \eqref{gi} and  mean value theorem one has 

\begin{align*}
-\Delta(v_{\lambda}- \underline{v})&\geq a(x) ((g(\theta_{1}(x)))^{-\gamma}g^{\prime}(\theta_{1}(x)))^{\prime}(v_{\lambda}-\underline{v}) + \lambda (g(v_{\lambda}))^{p}g^{\prime}(v_{\lambda})\\
&\geq -c_{1}(v_{\lambda}-\underline{v})
\end{align*}
in $\Omega$, because $\underline{v}(x)\leq\theta_{1}(x)\leq v_{\lambda}(x)$, and since $v_{\lambda}-\underline{v}\neq 0$, we can apply Theorem 3 of \cite{BN} to deduce the existence of constants $c_{3},c_{4}>0$ such that
$$v_{\lambda}-\underline{v}\geq c_{3}d(x)\geq c_{4}\phi_{1}(x)~\mbox{in}~\Omega.$$

As a consequence we obtain
$$\dfrac{\partial(v_{\lambda}-\underline{v})}{\partial \nu}\leq c_{4}\dfrac{\partial \phi_{1}}{\partial \nu}<0~\mbox{on}~\partial \Omega,$$
which jointly with $v_{\lambda}-\underline{v}>0$  in $\Omega$ means that  $v_{\lambda}-\underline{v}\in int (C_{0}^{1}(\overline{\Omega})_{+})$, and this proves the claim 2. \\
${\it Claim ~3}$. There exists a ball $B=B_{\epsilon}(v_{\lambda})$ in the $C^{1}_{0}(\overline{\Omega})$ topology satisfying 
$$B \subset [\underline{v},\overline{v}]:=\left\{v\in C^{1}_{0}(\overline{\Omega}):\underline{v}\leq v\leq \overline{v}~\mbox{in}~\Omega\right\}. $$
From claims 1 and 2 there exists $\epsilon>0$ such that the balls $B_{1}=B_{\epsilon}(\overline{v}-v_{\lambda}), B_{2}=B_{\epsilon}(v_{\lambda}-\underline{v})\subset int (C_{0}^{1}(\overline{\Omega})_{+})$. We define $B=B_{\epsilon}(v_{\lambda})$. Let $v\in B$. Notice that
$$\overline{v}-B_{1}=B_{\epsilon}(v_{\lambda})~\mbox{and}~\underline{v}+B_{2}=B_{\epsilon}(v_{\lambda}), $$
and as a consequence there exist $z\in B_{1},w\in B_{2}$ with 
$$\underline{v}+w=v=\overline{v}-z,$$
which implies $\underline{v}<v<\overline{v}$ in $\Omega$, that is, $v \in [\underline{v},\overline{v}]$. Hence $B \subset [\underline{v},\overline{v}]$.

We can finally complete the proof of the lemma. Let $B$ as in Claim 3 and consider $v\in B$. Then, 
\begin{align*}
J_{\lambda}(v)-I_{\lambda}(v)&=-\frac{\lambda}{p+1} \displaystyle \int_{\Omega}g^{p+1}(v)+\lambda \displaystyle \int_{\Omega}\tilde{G}(x,v)\\
&=-\frac{\lambda}{p+1} \displaystyle \int_{\Omega}g^{p+1}(v)+\lambda \displaystyle \int_{\Omega} \int_{0}^{\underline{v}(x)}\tilde{g}(x,t)dtdx+\lambda \displaystyle \int_{\Omega} \int_{\underline{v}(x)}^{v(x)}\tilde{g}(x,t)dtdx\\
&=-\frac{\lambda}{p+1} \displaystyle \int_{\Omega}g^{p+1}(v)+\lambda \displaystyle \int_{\Omega} \int_{0}^{\underline{v}(x)}g^{p}(\underline{v}(x))g^{\prime}(\underline{v}(x))dtdx\\
&+\lambda \displaystyle \int_{\Omega} \int_{\underline{v}(x)}^{v(x)}g^{p}(t)g^{\prime}(t)dtdx\\
&=\lambda \displaystyle \int_{\Omega} g^{p}(\underline{v}(x))g^{\prime}(\underline{v}(x))\underline{v}(x)dx-\frac{\lambda}{p+1} \displaystyle \int_{\Omega}g^{p+1}(\underline{v}(x))dx:=c
\end{align*}
where $c$ is a constant. 

By virtue of the above equality, we obtain that $v_{\lambda}$ is a $C_{0}^{1}(\overline{\Omega})$-local minimizer of $J_{\lambda}$. This finishes the proof.
\end{proof}
\begin{remark}
	Since $\underline{v},\overline{v}\in int(C_{0}^{1}(\overline{\Omega})_{+})$, it follows that $[\underline{v},\overline{v}]\subset int(C_{0}^{1}(\overline{\Omega})_{+})$ and then, by Remark \ref{rem2}, $J_{\lambda}(v),I_{\lambda}(v)\in \mathbb{R}$ for all $v\in [\underline{v},\overline{v}]$. Furthermore, arguing as in Lemma \ref{L000} we infer
	$$\left\{v\in H_{0}^{1}(\Omega):\underline{v}\leq v\leq \overline{v}~\mbox{in}~\Omega\right\}\subset D.$$

\end{remark}
\begin{corollary}\label{c30}
	Let $B=B_{\epsilon}(0)+v_{\lambda}$ be as in the proof of Lemma \ref{minl}. Then for all $v\in B_{\epsilon}(0)$ we have 
	$$J_{\lambda}(v_{\lambda}+v^{+})-J_{\lambda}(v_{\lambda})\geq 0,$$
	holds.
\end{corollary}
\begin{proof}
	As we have seen in the proof of Lemma \ref{minl},
	\begin{equation}\label{30}
	\underline{v}<v_{\lambda}+v<\overline{v}~\mbox{in}~\Omega,
	\end{equation}
for all $v\in B_{\epsilon}(0)$. We claim that 
$$	\underline{v}<v_{\lambda}+v^{+}<\overline{v}~\mbox{in}~\Omega,$$
for all $v\in B_{\epsilon}(0)$. Indeed, by using \eqref{30} one has
$$\underline{v}<v_{\lambda}+v=v_{\lambda}+v^{+}-v^{-}\leq v_{\lambda}+v^{+}~\mbox{in}~\Omega.$$

Now, let us show that $v_{\lambda}+v^{+}<\overline{v}$ in $\Omega$. Arguing by contradiction, suppose that there exists $x\in \Omega$ such that $v_{\lambda}(x)+v^{+}(x)\geq \overline{v}(x)$. Then, from $v_{\lambda}(x)<\overline{v}(x)$ we infer that $v(x)>0$, and therefore $v^{-}(x)=0$. Thus, the inequality \eqref{30} implies
$$\overline{v}(x)\leq v_{\lambda}(x)+v^{+}(x)=v_{\lambda}(x)+v^{+}(x)-v^{-}(x)=v_{\lambda}(x)+v(x)<\overline{v}(x),$$
a contradiction.

Finally, we can argue as in the proof of Lemma \ref{minl} to get 
$$J_{\lambda}(v_{\lambda}+v^{+})-I_{\lambda}(v_{\lambda}+v^{+})=c,$$
where $c$ is a constant, and since $v_{\lambda}+v^{+}\in H^{1}_{0}(\Omega)$, by Theorem \ref{TS}, we deduce that
$$J_{\lambda}(v_{\lambda}+v^{+})-J_{\lambda}(v_{\lambda})=I_{\lambda}(v_{\lambda}+v^{+})-I_{\lambda}(v_{\lambda})\geq 0,$$
proving the corollary.

\end{proof}

For fixed $\lambda \in (0,\lambda_{\ast})$, we look for a second solution in the form $z=w+v$, where $v\gneqq 0$ and $w=v_{\lambda}$ is the solution found in the preceding lemma. A straight calculation shows that $v$ satisfies
\begin{align}\label{pm}
-\Delta v&=a(x)( (g(w+v))^{-\gamma}g^{\prime}(w+v)- (g(w))^{-\gamma}g^{\prime}(w)) \\
&+ \lambda((g(w+v))^{p}g^{\prime}(w+v)-(g(w))^{p}g^{\prime}(w)).\nonumber
\end{align}
Denote by $g_{\lambda}(x,t)$ the right hand side of the preceding equation (with $g_{\lambda}(x,t)=0$ for $t\leq 0$) and set
\begin{equation}\label{J}
\mathcal{J}_{\lambda}(v)=\frac{1}{2}\Vert v\Vert^{2}\displaystyle -\displaystyle \int_{\Omega}G_{\lambda}(x,v), 
\end{equation}
where 

\begin{equation*}
	G_{\lambda}(x,t)=\displaystyle \int_{0}^{t}g_{\lambda}(x,s)ds=\left\{
	\begin{array}{l}
		0 ~\mbox{if} ~t\leq 0,\\
		H_{1}(x,t)+H_{2}(x,t)+H_{3}(x,t) ~\mbox{if} ~ t\geq 0,
	\end{array}
	\right.
\end{equation*}
and
$$H_{1}(x,t)=a(x)(G(w+t)-G(w)),$$
$$H_{2}(x,t)=\dfrac{\lambda}{p+1}(g^{p+1}(w+t)-g^{p+1}(w)),$$
$$H_{3}(x,t)=-a(x)g^{-\gamma}(w)g^{\prime}(w)t-\lambda(g(w))^{p}g^{\prime}(w)t,$$

for $t\geq 0$.

 We observe that by $(H)_{\infty}$, Lemma \ref{L1} $(3),(6),(11)$ and \eqref{41} one has
\begin{equation}\label{45}
\vert g_{\lambda}(x,t)\vert \leq c_{1}+2^{p/4}\lambda c_{2}\vert t\vert^{p/2},
\end{equation}
where $c_{1},c_{2}>0$ are constants which depends of $\Vert ag^{-\gamma}(w)g^{\prime}(w) \Vert_{\infty}$, $\Vert w\Vert_{\infty}$ and $p$. From this it follows that $\mathcal{J}_{\lambda}\in C^{1}(H_{0}^{1}(\Omega),\mathbb{R})$. 

We shall use the Mountain Pass Theorem to prove the existence of a nontrivial solution to \eqref{pm}. In order to do this, we need some preliminary lemmas.
\begin{lemma}\label{L32}
	$v=0$ is a local minimum of $\mathcal{J}_{\lambda}$ in $H_{0}^{1}(\Omega)$. 
\end{lemma}
\begin{proof}
	We write $v=v^{+}-v^{-}$. Using the fact that $w$ is a solution of \eqref{pa} and $G(x,t)=0$ for $t\leq 0$ we get
	\begin{align*}
	\mathcal{J}_{\lambda}(v)&=\frac{1}{2}\Vert v^{+}\Vert^{2}+\frac{1}{2}\Vert v^{-}\Vert^{2}\displaystyle -\displaystyle \int_{\Omega}G_{\lambda}(x,v^{+})+\frac{1}{2}\Vert w+ v^{+}\Vert^{2}-\frac{1}{2}\Vert w+v^{+}\Vert^{2}\\
	&=\frac{1}{2}\Vert v^{-}\Vert^{2}-\int_{\Omega}\nabla w\nabla v^{+}+\int_{\Omega}  a(x) (g(w))^{-\gamma}g^{\prime}(w)v^{+} +\lambda\int_{\Omega}(g(w))^{p}g^{\prime}(w)v^{+}\\
	&+\frac{1}{2}\Vert w+ v^{+}\Vert^{2}-\displaystyle \int_{\Omega}a(x) G(w+ v^{+})-\frac{\lambda}{p+1} \displaystyle \int_{\Omega}g^{p+1}(w+v^{+})\\
	&-\frac{1}{2}\Vert w\Vert^{2}+ \displaystyle \int_{\Omega}a(x) G(w)+\frac{\lambda}{p+1} \displaystyle \int_{\Omega}g^{p+1}(w)\\
	&=\frac{1}{2}\Vert v^{-}\Vert^{2}+J_{\lambda}(w+v^{+})-J_{\lambda}(w).
	\end{align*}
This and Corollary \ref{c30} imply that $\mathcal{J}_{\lambda}(v)\geq 0$ for all $v\in B_{\epsilon}(0)$, where $B_{\epsilon}(0)$ is as in Corollary \ref{c30}. This proves that $v=0$ is a local minimum in the $C^{1}_{0}(\overline{\Omega})$ topology. Therefore, in view of \eqref{45}, Theorem 1 in \cite{BN} applies and $v=0$ is a local minimum of $\mathcal{J}_{\lambda}$ in the $H_{0}^{1}(\Omega)$ topology. This finishes the proof.

\end{proof}

\begin{lemma}\label{L300}
	If $v,w\in L^{\infty}(\Omega)\cap D$ are positive functions, then 
	$$\displaystyle \lim_{t\to \infty}\int_{\Omega}\dfrac{a(x)G(v+tw)}{t^{(p+1)/2}}=0$$
	and 
	$$\displaystyle \int_{\Omega}g^{p+1}(v+tw)\geq t^{(p+1)/2}\int_{\Omega}g^{p+1}(\dfrac{v}{t}+w),$$
	for all $t>1$. 
\end{lemma}
\begin{proof}
	First we prove the limit. We divide the proof into three cases.\\
	{\it Case 1.} $\gamma < 1$. In this case, by Lemma \ref{L1} $(5)$ one has
	$$0<\dfrac{a(x)G(v+tw)}{t^{(p+1)/2}}=\dfrac{a(x)g^{1-\gamma}(v+tw)}{(1-\gamma)t^{(p+1)/2}}\leq\dfrac{a(x)(\dfrac{v}{t}+w)^{1-\gamma}}{(1-\gamma)t^{((p+1)/2)+\gamma-1}}\leq \dfrac{a(x)(v+w)^{1-\gamma}}{1-\gamma}, $$
	for all $t\geq 1$. Then taking the limit as $t\to \infty$ we get
	$$\dfrac{a(x)G(v+tw)}{t^{(p+1)/2}}\to 0,$$
	and from the Lebesgue dominated convergence theorem we find
	$$\displaystyle \lim_{t\to \infty}\int_{\Omega}\dfrac{a(x)G(v+tw)}{t^{(p+1)/2}}=0.$$
	This proves the case 1. \\
	{\it Case 2.} $\gamma=1$. By Lemma \ref{L1} $(3),(5)$
	\begin{align*}
	\dfrac{a(x)ln(g(v))}{t^{(p+1)/2}}\leq\dfrac{a(x)G(v+tw)}{t^{(p+1)/2}}=\dfrac{a(x)ln(g(v+tw))}{t^{(p+1)/2}}\leq\dfrac{a(x)(\dfrac{v}{t}+w)}{t^{((p+1)/2)-1}}\leq a(x)(v+w)
	\end{align*}
	for all $t\geq 1$, and thus 
	
	$$\vert \dfrac{a(x)G(v+tw)}{t^{(p+1)/2}}\vert \leq \max \left\{\vert a(x)ln(g(v))\vert,a(x)(v+w)  \right\}.$$
	
	 Again, by the  Lebesgue dominated convergence theorem we have
	$$\displaystyle \lim_{t\to \infty}\int_{\Omega}\dfrac{a(x)G(v+tw)}{t^{(p+1)/2}}=0.$$\\
	{\it Case 3.} $\gamma >1$.  By Lemma \ref{L1} $(3),(10)$ one has
	$$0<\vert \dfrac{a(x)G(v+tw)}{t^{(p+1)/2}}\vert =\dfrac{a(x)g^{1-\gamma}(v+tw)}{\vert 1-\gamma \vert t^{(p+1)/2}}\leq\dfrac{a(x)g^{1-\gamma}(v)}{\vert 1-\gamma\vert t^{(p+1)/2}}\leq \dfrac{a(x)g^{1-\gamma}(v)}{\vert 1-\gamma\vert }, $$
	for all $t\geq 1$. By the  Lebesgue dominated convergence theorem one finds
	$$\displaystyle \lim_{t\to \infty}\int_{\Omega}\dfrac{a(x)G(v+tw)}{t^{(p+1)/2}}=0.$$

	We now fix $t>1$. Then, from Lemma \ref{L1} $(13)$ we have
	$$g^{p+1}(v+tw)= \left[g^{2}(t(\dfrac{v}{t}+w))\right]^{(p+1)/2}\geq\left[tg^{2}(\dfrac{v}{t}+w)\right]^{(p+1)/2} =t^{(p+1)/2}g^{p+1}(\dfrac{v}{t}+w),$$
	and this implies that
	$$\displaystyle \int_{\Omega}g^{p+1}(v+tw)\geq t^{(p+1)/2}\int_{\Omega}g^{p+1}(\dfrac{v}{t}+w),$$
	for all $t>1$. The lemma is proved.
\end{proof}

\begin{lemma}\label{sps}
	Let $2<\theta<p+1$. Then, for all $t\geq 0$,
	\begin{itemize}
		\item[a)] $-G_{\lambda}(x,t)+\dfrac{\theta}{p+1}g_{\lambda}(x,t)t\geq c-\dfrac{a(x)}{1-\gamma}t^{1-\gamma} $ for some constant $c\in \mathbb{R}$, if $0<\gamma<1$;
		\item[b)] $-G_{\lambda}(x,t)+\dfrac{\theta}{p+1}g_{\lambda}(x,t)t\geq c-a(x)t +a(x)ln(g(w))$ for some constant $c\in \mathbb{R}$, if $\gamma=1$;
		\item[c)] $-G_{\lambda}(x,t)+\dfrac{\theta}{p+1}g_{\lambda}(x,t)t\geq c+\dfrac{a(x)}{1-\gamma}g^{1-\gamma}(w) $ for some constant $c\in \mathbb{R}$, if $\gamma>1$.
	\end{itemize}
\end{lemma}
\begin{proof}
	For convenience of notation we write
	$$h_{1}(x,t)=a(x) (g(w+t))^{-\gamma}g^{\prime}(w+t),$$
	$$h_{2}(x,t)=\lambda(g(w+t))^{p}g^{\prime}(w+t),$$
	$$h_{3}(x,t)=-a(x) (g(w))^{-\gamma}g^{\prime}(w)-\lambda(g(w))^{p}g^{\prime}(w),$$
	for $t\geq 0$. Thus, 
	\begin{align*}
	-G_{\lambda}(x,t)+\dfrac{\theta}{p+1}g_{\lambda}(x,t)t&=-H_{1}(x,t)+\dfrac{\theta}{p+1}h_{1}(x,t)t\\
	&-H_{2}(x,t)+\dfrac{\theta}{p+1}h_{2}(x,t)t\\
	&-H_{3}(x,t)+\dfrac{\theta}{p+1}h_{3}(x,t)t
	\end{align*}
for $t\geq 0$. \\
a) In this case, from Lemma \ref{L1} $(5)$	we have
\begin{align*}
-H_{1}(x,t)+\dfrac{\theta}{p+1}h_{1}(x,t)t\geq -H_{1}(x,t)&\geq -\dfrac{a(x)}{1-\gamma}g^{1-\gamma}(w+t)\\
&\geq -\dfrac{a(x)}{1-\gamma}(w+t)^{1-\gamma}\\
&\geq -\dfrac{a(x)}{1-\gamma}(w^{1-\gamma}+t^{1-\gamma})\\
&\geq -\dfrac{\Vert aw^{1-\gamma}\Vert_{\infty}}{1-\gamma}-\dfrac{a(x)}{1-\gamma}t^{1-\gamma}
\end{align*}
and since $p+1>\theta$, 
\begin{equation}\label{32}
-H_{3}(x,t)+\dfrac{\theta}{p+1}h_{3}(x,t)t=(1-\dfrac{\theta}{p+1})(a(x) (g(w))^{-\gamma}g^{\prime}(w)+\lambda(g(w))^{p}g^{\prime}(w))t\geq 0.
\end{equation}
Let us observe that this inequality is valid for all $\gamma>0$. 

Now, let us estimate $-H_{2}(x,t)+\dfrac{\theta}{p+1}h_{2}(x,t)t$. From Lemma \ref{L1} $(4)$ one has 
\begin{align*}
-H_{2}(x,t)+\dfrac{\theta}{p+1}h_{2}(x,t)t&=\dfrac{\lambda}{p+1}(-g^{p+1}(w+t)+g^{p+1}(w))\\
&+\dfrac{\theta \lambda}{p+1}(g(w+t))^{p}g^{\prime}(w+t)t\\
&\geq \dfrac{\lambda}{p+1}\left[-g^{p+1}(w+t)+\dfrac{\theta }{2}\dfrac{g^{p+1}(w+t)t}{w+t}+g^{p+1}(w)\right]\\
&\geq \dfrac{\lambda}{p+1}\left[g^{p+1}(w+t)(-1+\dfrac{\theta }{2}\dfrac{t}{\Vert w\Vert_{\infty}+t})+g^{p+1}(w)\right],
\end{align*}
and therefore 
$$-H_{2}(x,t)+\dfrac{\theta}{p+1}h_{2}(x,t)t>0,$$
for all $t>\overline{t}:=(2\Vert w \Vert_{\infty})/(\theta-2)$. 

Moreover, for $0\leq t\leq \overline{t}$,  by using Lemma \ref{L1} $(5)$ we get
\begin{align*}
-H_{2}(x,t)+\dfrac{\theta}{p+1}h_{2}(x,t)t&\geq -\dfrac{\lambda}{p+1}g^{p+1}(w+t)\\
&\geq -\dfrac{\lambda}{p+1}(w+t)^{p+1}\geq -\dfrac{\lambda}{p+1}(\Vert w \Vert_{\infty}+\overline{t})^{p+1}.
\end{align*}

By setting $c_{1}=-\dfrac{\lambda}{p+1}(\Vert w \Vert_{\infty}+\overline{t})^{p+1}$, we have proved that 
\begin{equation}\label{33}
-H_{2}(x,t)+\dfrac{\theta}{p+1}h_{2}(x,t)t\geq c_{1}, ~\mbox{for all}~t\geq 0.
\end{equation}
Let us observe that this inequality is valid independent of $\gamma>0$.

In view  of the above inequalities we deduce that
$$-G_{\lambda}(x,t)+\dfrac{\theta}{p+1}g_{\lambda}(x,t)t\geq c-\dfrac{a(x)}{1-\gamma}t^{1-\gamma} ~\mbox{for all}~t\geq 0,$$
where $c=-\dfrac{\Vert aw^{1-\gamma}\Vert_{\infty}}{1-\gamma}+c_{1}$.\\
b) When $\gamma=1$, by Lemma \ref{L1} $(5)$,  one has the inequality
\begin{align*}
-H_{1}(x,t)+\dfrac{\theta}{p+1}h_{1}(x,t)t&\geq -H_{1}(x,t)= -a(x)ln(g(w+t))+a(x)ln(g(w))\\
&\geq -a(x)(w+t)+a(x)ln(g(w))\\
&\geq -\Vert aw\Vert_{\infty}-a(x)t+a(x)ln(g(w)),
\end{align*}
which combined with \eqref{32} and \eqref{33} yield
$$-G_{\lambda}(x,t)+\dfrac{\theta}{p+1}g_{\lambda}(x,t)t\geq c-a(x)t +a(x)ln(g(w))$$
 for some constant $c\in \mathbb{R}$.\\
c) Indeed, the inequality
\begin{align*}
-H_{1}(x,t)+\dfrac{\theta}{p+1}h_{1}(x,t)t\geq -H_{1}(x,t)&=-\dfrac{a(x)}{1-\gamma}g^{1-\gamma}(w+t)+\dfrac{a(x)}{1-\gamma}g^{1-\gamma}(w)\\
&\geq \dfrac{a(x)}{1-\gamma}g^{1-\gamma}(w),
\end{align*}
combined with \eqref{32} and \eqref{33} yield 
$$-G_{\lambda}(x,t)+\dfrac{\theta}{p+1}g_{\lambda}(x,t)t\geq c+\dfrac{a(x)}{1-\gamma}g^{1-\gamma}(w) $$
 for some constant $c\in \mathbb{R}$. This concludes the proof.

\end{proof}

We are now ready to prove Theorem \ref{T2}.

\textbf{Proof of Theorem \ref{T2}} By Lemma \ref{L32} $u_{0}=0$ is a local minimizer of $\mathcal{J}_{\lambda}$ with respect to the topology of $H_{0}^{1}(\Omega)$. In the case where $u_{0}$ is not a strict local minimizer of $\mathcal{J}_{\lambda}$, we deduce the existence of further critical points of $\mathcal{J}_{\lambda}$, and then we are done. In this way, we may assume that 
\begin{equation}\label{GM1}
u_{0}=0~\mbox{is a strict local minimizer of}~\mathcal{J}_{\lambda}.
\end{equation}

 For all $t>1$ we have
$$\mathcal{J}_{\lambda}(t\phi_{1})=J_{\lambda}(w+t\phi_{1})-J_{\lambda}(w)$$
and by Lemma \ref{L300} it follows that 

\begin{align*}
\mathcal{J}_{\lambda}(t\phi_{1})\leq \frac{1}{2}\Vert w+t\phi_{1}\Vert^{2}-\displaystyle \int_{\Omega}a(x) G( w+t\phi_{1})-\frac{t^{(p+1)/2}\lambda}{p+1} \int_{\Omega}g^{p+1}(\dfrac{w}{t}+\phi_{1})-J_{\lambda}(w)
\end{align*}
and using again Lemma \ref{L300} and the Lebesgue dominated convergence theorem we yield $\displaystyle \lim_{t \to \infty}\mathcal{J}(t\phi_{1})=-\infty$. From this and \eqref{GM1}, we conclude that $\mathcal{J}_{\lambda}$ has the  mountain pass geometry (see \cite{ARY}, Theorem 2.1). It remains to prove the Palais-Smale condition. Let $4<2\theta<p+1$. Let $v_{n}\in H_{0}^{1}(\Omega)$ be such that $\mathcal{J}_{\lambda}(v_{n})\to c$ ($c\in \mathbb{R}$) and $\mathcal{J}^{\prime}_{\lambda}(v_{n})\to 0$. From the former, respectively the latter multiplied by $\theta v_{n}/(p+1)$, we get
$$\frac{1}{2}\Vert v_{n}\Vert^{2}\displaystyle -\displaystyle \int_{\Omega}G_{\lambda}(x,v_{n})=c+o(1),  $$

$$o(1)\Vert v_{n}\Vert \geq \vert \dfrac{\theta}{p+1}\Vert  v_{n}\Vert^{2}\displaystyle -\dfrac{\theta}{p+1}\displaystyle \int_{\Omega}g_{\lambda}(x,v_{n})v_{n} \vert\geq \dfrac{-\theta}{p+1}\Vert  v_{n}\Vert^{2}\displaystyle +\dfrac{\theta}{p+1}\displaystyle \int_{\Omega}g_{\lambda}(x,v_{n})v_{n},$$
and therefore (remember that $G_{\lambda}(x,t)=g_{\lambda}(x,t)t=0$ for $t\leq0$), 
\begin{align*}
c+o(1)+o(1)\Vert v_{n}\Vert \geq (\dfrac{1}{2}-\dfrac{\theta}{p+1})\Vert  v_{n}\Vert^{2}+\displaystyle \int_{\Omega}(-G_{\lambda}(x,v^{+}_{n})+\dfrac{\theta}{p+1}g_{\lambda}(x,v^{+}_{n})v^{+}_{n}).  
\end{align*}
From this and Lemma \ref{sps} we deduce that
\begin{equation*}
c+o(1)+o(1)\Vert v_{n}\Vert \geq\left\{
\begin{array}{l}
(\dfrac{1}{2}-\dfrac{\theta}{p+1})\Vert  v_{n}\Vert^{2}+ c\vert \Omega \vert-\displaystyle \int_{\Omega}\dfrac{a(x)}{1-\gamma}(v_{n}^{+})^{1-\gamma}  ~\mbox{if} ~ \gamma <1,\\
(\dfrac{1}{2}-\dfrac{\theta}{p+1})\Vert  v_{n}\Vert^{2}+ c\vert \Omega \vert-\displaystyle \int_{\Omega}a(x)v_{n}^{+} +\displaystyle \int_{\Omega}a(x)ln(g(w))\\
 ~\mbox{if} ~ \gamma=1,\\
(\dfrac{1}{2}-\dfrac{\theta}{p+1})\Vert  v_{n}\Vert^{2}+c\vert \Omega \vert +\displaystyle \int_{\Omega}\dfrac{a(x)}{1-\gamma}g^{1-\gamma}(w)  ~\mbox{if} ~  \gamma>1. 
\end{array}
\right.
\end{equation*}
Thus, in any  case, by the Sobolev embedding theorem we have that the sequence $\left\{v_{n}\right\}$ is bounded in $H_{0}^{1}(\Omega)$ and a standard argument shows that, up to a subsequence, there exists $v\in H_{0}^{1}(\Omega)$ such that $v_{n}\to v$ in $H_{0}^{1}(\Omega)$. Therefore, the Palais-Smale condition has been verified. 

Finally, an application of the mountain pass theorem  yields a nontrivial critical point $v$ of $\mathcal{J}_{\lambda}$ (see \cite{ARY}, Theorem 2.1) and by elliptic regularity $v\in C_{0}^{1}(\overline{\Omega})$. Moreover, since $g_{\lambda}(x,t)=0$ for $t\leq0$ one has $-\Vert v^{-}\Vert^{2}=0$, which implies that $v\gneqq 0$ and  $z=w+v\in C_{0}^{1}(\overline{\Omega})$ is a second solution of \eqref{pa}. This finishes the proof of Theorem \ref{T2}.

We end this section with the following proposition.
\begin{proposition}\label{p0}
Suppose that  $(H)_{\infty}$ and $3<p<22^{\ast}-1$ hold.	If $0<\gamma<1$, then $\lambda_{\ast}\in \mathcal{L}$.
\end{proposition}
\begin{proof}
	In order to prove the proposition one uses the following properties:
	\begin{itemize}
		\item if $v_{\lambda}$ is the solution obtained in Theorem \ref{T1}, then $J_{\lambda}(v_{\lambda})<c$ for some constant $c>0$ independent of $\lambda \in (0,\lambda_{\ast})$. Indeed, as we have seen in the proof of Lemma \ref{minl},
		$$J_{\lambda}(v_{\lambda})=I_{\lambda}(v_{\lambda})+\lambda \displaystyle \int_{\Omega} g^{p}(\underline{v}(x))g^{\prime}(\underline{v}(x))\underline{v}(x)dx-\frac{\lambda}{p+1} \displaystyle \int_{\Omega}g^{p+1}(\underline{v}(x))dx,$$
		
which, jointly with Remark \ref{rem2} $b)$, gives
$$J_{\lambda}(v_{\lambda})\leq \lambda_{\ast} \displaystyle \int_{\Omega} g^{p}(\underline{v}(x))g^{\prime}(\underline{v}(x))\underline{v}(x)dx:=c.	$$	
	\item $\dfrac{-g^{p+1}(t)}{p+1}+\dfrac{\theta}{p+1}g^{p}(t)g^{\prime}(t)t\geq 0$ for all $t>0$ and $4<2\theta<p+1$. Indeed, from Lemma \ref{L1} $(4)$ we get 
		\begin{equation}\label{fin}
		\dfrac{-g^{p+1}(t)}{p+1}+\dfrac{\theta}{p+1}g^{p}(t)g^{\prime}(t)t\geq \dfrac{g^{p+1}(t)}{p+1}\left(-1+\dfrac{\theta}{2}\right)>0 ,
		\end{equation}
		for all $t>0$. 
	\end{itemize}
	Now, let $\lambda_{n}\in (0,\lambda_{\ast})$ be an increasing sequence such that $\lambda_{n}\to \lambda_{\ast}$ as $n\to \infty$ and let $v_{n}:=v_{\lambda_{n}}$ be a solution of \eqref{pa} obtained in Theorem \ref{T1} for $\lambda=\lambda_{n}$. Then 
	$$J_{\lambda_{n}}(v_{n})=\frac{1}{2}\Vert v_{n}\Vert^{2}-\displaystyle \int_{\Omega}a(x) G( v_{n})-\frac{\lambda_{n}}{p+1} \displaystyle \int_{\Omega}g^{p+1}(v_{n})<c,$$
for some constant $c>0$ independent of $\lambda_{n}$	and
	$$\Vert v_{n}\Vert^{2}-\displaystyle \int_{\Omega}a(x) (g(v_{n}))^{-\gamma}g^{\prime}(v_{n})v_{n} - \lambda_{n}\displaystyle \int_{\Omega}(g(v_{n}))^{p}g^{\prime}(v_{n})v_{n} =0.$$
	
	Thus, by using \eqref{fin}, one deduces 
	$$\left(\frac{1}{2}-\dfrac{\theta}{p+1}\right)\Vert v_{n}\Vert^{2}-\dfrac{1}{1-\gamma}\displaystyle \int_{\Omega}a(x)  g^{1-\gamma}(v_{n})+\dfrac{\theta}{p+1}\displaystyle \int_{\Omega}a(x) (g(v_{n}))^{-\gamma}g^{\prime}(v_{n})v_{n} <c,$$
	whence, by Lemma \ref{L1} $(3)$,
	$$\left(\frac{1}{2}-\dfrac{\theta}{p+1}\right)\Vert v_{n}\Vert^{2}<\dfrac{1}{1-\gamma}\displaystyle \int_{\Omega}a(x)  g^{1-\gamma}(v_{n})+c\leq \dfrac{\Vert a\Vert_{\infty}}{1-\gamma}\displaystyle \int_{\Omega}  v_{n}^{1-\gamma}+c.$$
	From the previous relation it is easy to see that $\left\{v_{n}\right\}$ is bounded in $H_{0}^{1}(\Omega)$. Thus, there exists $v^{\ast}\in H_{0}^{1}(\Omega)$ such that, up to a subsequence, we have as $n \to \infty$
	$$v_{n}\rightharpoonup v^{\ast}~\mbox{in}~H_{0}^{1}(\Omega),$$
	$$v_{n}\to v^{\ast}~\mbox{a.e. in}~ \Omega.$$
	
	Remember that $v_{n}\geq \underline{v}=v_{0}$ in $\Omega$ and thus, by Lemma \ref{L1} $(9),(11)$, $\vert a(x) (g(v_{n}))^{-\gamma}g^{\prime}(v_{n})\psi \vert \leq  \vert a(x) (g(v_{0}))^{-\gamma}g^{\prime}(v_{0})\psi \vert $ in $\Omega$.

	Because $v_{n}$ is a solution of $(Q_{\lambda_{n}})$, we have
$$
\int_{\Omega} \nabla v_{n} \nabla \psi =\int_{\Omega}  a(x) (g(v_{n}))^{-\gamma}g^{\prime}(v_{n})\psi +\lambda_{n}\int_{\Omega} (g(v_{n}))^{p}g^{\prime}(v_{n})\psi,$$
for all $\psi \in H_{0}^{1}(\Omega)$. Passing to the limit in the previous equality and using Lebesgue's theorem, we deduce that $v^{\ast}$ is a weak solution of $(Q_{\lambda_{\ast}})$. Finally, we can adapt the arguments in the proof of Theorem 1 $c)$ in \cite{R} to obtain $v^{\ast}\in C^{1}_{0}(\overline{\Omega})$. This ands the proof of the proposition.
\end{proof}
Proposition \ref{p0} suggests that $\lambda_{\ast}\in \mathcal{L}$ for arbitrary $\gamma>0$. However, for $\gamma>1$ and $\lambda \in (0,\lambda_{\ast})$ one has $J_{\lambda}(v)>0$ for any solution $v$ of \eqref{pa}, and thus the proof of Proposition \ref{p0} cannot be applied to deduce that $\lambda_{\ast}\in \mathcal{L}$.

{\it\small  Ricardo Lima Alves}\\{\it\small  Unidade Acad\^emica de Matem\'atica}\\
	{\it \small Universidade Federal de Campina Grande}\\
	{\it \small CEP: 58429-900, Campina Grande, PB-Brazil}\\{\it\small e-mail: ricardoalveslima8@gmail.com}

\end{document}